\documentclass[10pt,a4paper,reqno]{amsart} 

\usepackage{amsmath}
\usepackage{mathtools}
\usepackage{amssymb}
\usepackage{graphicx}
\usepackage{color}
\usepackage{latexsym}
\usepackage[T1]{fontenc}

\newcommand{\ns}{{\mathbb N}} 
\newcommand{\zs}{{\mathbb Z}} 
\newcommand{\qs}{{\mathbb Q}}  

\newcommand{\al}{\alpha}

\newcommand{\si}{\sigma}
\newcommand{\la}{\lambda}
\newcommand{\eps}{\epsilon}
\newcommand{\vareps}{\varepsilon}

\newcommand{\bx}{\bar x}

\DeclareMathOperator{\id}{id}

\DeclareMathOperator{\SUM}{SUM}
\newcommand{\cE}{\mathcal E}

\newtheorem{Theorem}{Theorem}
\newtheorem{Proposition}[Theorem]{Proposition}

\newtheorem{Definition}[Theorem]{Definition}
\newtheorem{Lemma}[Theorem]{Lemma}
\newtheorem{Property}{Property}

\newcommand{\beq}{\begin{equation}}
\newcommand{\eeq}{\end{equation}}

\newcommand{\gf}{generating function}
\newcommand{\gfs}{generating functions}
\newcommand{\fps}{formal power series}

\def\emm#1,{{\em #1}}

\newcommand{\p}{permutation}
\newcommand{\ps}{permutations}
\newcommand{\inv}{involution}
\newcommand{\invs}{involutions}

\newcommand{\incm}{$12\cdots m(m+1)$}
\newcommand{\decm}{$ (m+1)m\cdots 21$}

\newcommand{\Sn}{{\mathfrak S}}
\newcommand{\In}{{\mathfrak I}}

\newcommand{\Fa}{\tilde F}

\newcommand{\Fc}{F}
\newcommand{\Fd}{\bar F}
\newcommand{\Fe}{F}

\newcommand{\Ga}{\tilde G}

\newcommand{\Gc}{G}
\newcommand{\Gd}{\bar G}
\newcommand{\Ge}{G}

\newcommand{\B}{B}

\graphicspath{{Figures/}}

\catcode`\@=11
\def\section{\@startsection{section}{1}%
 \z@{.7\linespacing\@plus\linespacing}{.5\linespacing}%
 {\normalfont\bfseries\scshape\centering}}

\def\subsection{\@startsection{subsection}{2}%
  \z@{.5\linespacing\@plus\linespacing}{.5\linespacing}%
  {\normalfont\bfseries\scshape}}

\def\subsubsection{\@startsection{subsubsection}{3}%
 \z@{.5\linespacing\@plus\linespacing}{-.5em}
  {\normalfont\bfseries\itshape}}
\catcode`\@=12

\def\qed{$\hfill{\vrule height 3pt width 5pt depth 2pt}$}

\begin{document}
\title
[Permutations with no long monotone subsequence]
{
Counting permutations\\
 with no long  monotone subsequence\\
via generating trees and the kernel method}

\author[M. Bousquet-M\'elou]{Mireille Bousquet-M\'elou}
\address{M. Bousquet-M\'elou: CNRS, LaBRI, Universit\'e Bordeaux 1, 
351 cours de la Lib\'eration, 33405 Talence, France}
\email{mireille.bousquet@labri.fr}
%



\begin{abstract}
We recover Gessel's determinantal formula for the \gf\ of \ps\ with no
ascending subsequence of length $m+1$. The starting point of our proof is the
recursive construction of these \ps\ by insertion of the largest
entry. This construction is of course extremely simple. The cost of this
simplicity is that we need to take into
account in the enumeration $m-1$ additional parameters --- namely, the positions
of the leftmost increasing subsequences of length $i$, for $i=2,\ldots,m$. This yields for the \gf\ a functional equation with $m-1$
``catalytic'' variables, and the heart of the
paper is the solution of this equation.

We perform a similar task for involutions with no descending subsequence
of length $m+1$, constructed recursively by adding a cycle containing
the largest entry. We refine this result by keeping track of the
number of fixed points.

In passing, we prove that the ordinary \gfs\ of these families of
\ps\ can be expressed as constant terms of rational series.
\end{abstract}
\maketitle

\date{\today}


\section{Introduction}
Let $\tau=\tau(1)\cdots \tau(n)$ be a \p\ in the symmetric group
$\Sn_{n}$. We denote by $|\tau|:=n$ the \emm length, of $\tau$.
An \emm ascending, (resp. \emm descending,)  subsequence of
$\tau$ of length $k$ is 
a $k$-tuple $(\tau(i_1), \ldots, \tau(i_k))$ such that $i_1 <\cdots <
i_k$ and $\tau(i_1)< \cdots < \tau(i_k)$ (resp. $\tau(i_1)> \cdots >
\tau(i_k)$). 
For $m\ge 1$, the
set of \ps\ in which all ascending subsequences have length at most $m$
  is denoted by $\Sn^{(m)}$.
In pattern-avoidance terms, the \ps\ of $\Sn^{(m)}$
are those that \emm avoid, the increasing pattern \incm, and an ascending
subsequence of length $k$ is an \emm occurrence, of the pattern
$12\cdots k$.
The set of
\incm-avoiding \ps\ of length $n$ is denoted $\Sn_n^{(m)}$.  Note that
several families of pattern avoiding \ps\ are equinumerous with
$\Sn_n^{(m)}$ (see~\cite{bwx,kratt-growth-diagrams}).

In 1990, Gessel proved a beautiful determinantal formula for what
could be called the \emm Bessel, \gf\ of \ps\ of $\Sn^{(m)}$.
This formula was the starting point of Baik, Deift and Johansson's
study of the distribution of the longest ascending subsequence in a
random \p\ \cite{baik-deift-johansson}.

\begin{Theorem}[\cite{gessel-symmetric}]\label{thm:p}
  The Bessel \gf\ of \ps \ avoiding \incm\ is
$$
  \sum_{\tau \in \Sn^{(m)}}\frac{t^{2|\tau|}}{|\tau|!^2} =
\det \left(I_{i-j}\right)_{1\le i,j\le m},
$$
where
\beq\label{Idef}
I_i=\sum_{n\ge \max(0, -i)} \frac{t^{2n+i}}{n! (n+i)!}.
\eeq
\end{Theorem}
Note that $I_i=I_{-i}$, and that we can more loosely write
$$
I_i=\sum_{n\ge 0} \frac{t^{2n+i}}{n! (n+i)!}=\sum_{n\ge 0} \frac{t^{2n-i}}{n! (n-i)!},
$$
provided we interpret factorials as Gamma functions (in particular,
$1/i!=1/\Gamma(i+1)=0$ if $i<0$).

Gessel's original proof was algebraic in nature~\cite{gessel-symmetric}. He first established 
a  determinantal identity dealing with Schur functions
(and hence with semi-standard Young tableaux, whereas the above theorem
deals, via Schensted's correspondence, with standard tableaux). He
then applied to this identity  an operator $\theta$ that extracts certain
coefficients, and this led to Theorem~\ref{thm:p}.
A few years later, Krattenthaler found a bijective proof of Gessel's
Schur function identity~\cite{krattenthaler-firenze},  which  specializes into a bijective
proof of Theorem~1. Then, Gessel, Weinstein and Wilf gave two bijective
proofs of this theorem, involving sign-reversing
involutions~\cite{gessel-weinstein-wilf}. Two other proofs, involving
Young tableaux, were recently published by Novak~\cite{novak} and
Xin~\cite{xin}.

For small values of $m$,   more proofs of
Theorem~\ref{thm:p} have been given.
In particular, there exists a
 wealth of ways of proving  that the number of 123-avoiding \ps\ of
 $\Sn_n$ is the $n^{\hbox{\small th}}$ Catalan number ${{2n} \choose
   n}/(n+1)$, and numerous refinements of this result~\cite{bloom-saracino,mbm-motifs,elizalde-pak,knuth3,kratti,robertson,robertson-saracino-zeilberger,rotem,west-catalan}. The laziest proof (combinatorially speaking) is based on
 the following observation: a 
 \p\ $\pi$ of $\Sn_{n+1}^{(2)}$ is obtained by inserting $n+1$ in a
 \p\ $\tau$ of  $\Sn_n^{(2)}$. To avoid the creation of an ascending
 subsequence of length 3, the insertion must not take place to the right
 of the leftmost ascent of $\tau$. Hence, in order to exploit this
 simple recursive description of \ps\ of $\Sn^{(2)}$, one must keep
 track of the position of the first ascent.  Let us denote
$$
a(\tau)=  
\left\{
\begin{array}{ll}
n+1, & \hbox{if } \tau \hbox{ avoids } 12;
\\
  \min\,\{i:  \tau(i-1)<  \tau(i)\},
& \hbox{otherwise,}
\end{array}
\right.
$$
and define the bivariate \gf\
$$
F(u;t):=\sum_{\tau\in \Sn^{(2)}} u^{a(\tau)-1}t^{|\tau|}.
$$
It is not hard to see (and this will be explained in details
  in Section~\ref{sec:gen}) that  the recursive description of
  \ps\ of $\Sn^{(2)}$ translates into the following equation:
\beq\label{eq:m=2}
\left(1-t\, \frac{u^2}{u-1}\right) F(u;t)= 
1-t\, \frac {u}{u-1}F(1;t).
\eeq
The variable $u$ is said to be \emm catalytic, for this equation. This
means that one cannot simply set $u=1$ to solve for $F(1;t)$ first.
However, this equation can be solved using the so-called \emm kernel method,
(see, e.g.,~\cite{hexacephale,bousquet-petkovsek-1,prodinger}): one
specializes $u$ to the unique power series $U$ that cancels the
 \emm kernel, of the equation (that is, the coefficient of $F(u;t)$):
$$
U:=\frac{1-\sqrt{1-4t}}{2t}.
$$
This choice cancels the
 left-hand side of the equation, and thus its right-hand side, yielding the (ordinary) length \gf\ of 123-avoiding \ps:
$$
F(1;t)= \frac {U-1}{t U}=U=\frac{1-\sqrt{1-4t}}{2t}=\sum_{n\ge 0}
\frac{t^n}{n+1}{{2n}\choose n}.
$$

It is natural to ask whether this approach can be generalized to a
generic value of $m$: after all,  a 
 \p\ $\pi$ of $\Sn_{n+1}^{(m)}$ is still obtained by inserting $n+1$ in a
 \p\ $\tau$ of  $\Sn_n^{(m)}$. However, to avoid creating an ascending
 subsequence of length $m+1$, the insertion must not take place to the
 right of the leftmost ascending subsequence of length $m$ of
 $\tau$. In order  to keep track \emm recursively, of the position 
of this
 subsequence, one must also keep track  of the position of the
 leftmost ascending subsequence of length $m-1$. And so on! Hence this
 recursive construction (often called the \emm generating tree,
 construction~\cite{west-catalan,west-gen})  translates into a functional equation involving
 $m-1$ catalytic variables $u_2, \ldots, u_{m}$. The whole point is
 to \emm solve, this equation, and this is what we do in this paper.
Our method combines three ingredients: an appropriate change of
variables, followed by what is essentially the reflection
principle~\cite{gessel-zeilberger}, but
performed at the level of power series, and finally a coefficient extraction. To warm up, we illustrate these 
ingredients in Section~\ref{sec:ex} by two simple examples: we first
give another solution of the equation~\eqref{eq:m=2} obtained when
$m=2$, and then a \gf\ proof of MacMahon's formula for the number of
standard tableaux of a given shape. 

\medskip
What is the interest of this exercise? Firstly, we believe it answers a
natural question: we have in one hand a simple recursive construction
of certain permutations, in the other hand a nice expression for
their \gf, and it would be frustrating not to be able  to derive the
expression from the construction. Secondly, the combinatorial
literature abunds in  objects that can be described
recursively by keeping track of an arbitrary (but bounded) number of
additional (or: catalytic) parameters: permutations of course, but also
lattice paths, tableaux, matchings, plane partitions, set partitions... Some, but not
all, can be solved by the reflection principle, and we hope that this
first solution of an equation with $m$ catalytic variables will be
followed by others.

In fact, we provide in this paper another application of our
approach, still in the field of \ps: We recover a determinantal formula
for the enumeration of \emm involutions, with no long descending
subsequence~\cite{gessel-symmetric}.
Let  $\In^{(m)}$ (resp. $\In_n^{(m)}$) denote the set of involutions
(resp. involutions of length $n$) avoiding 
the  \emm decreasing,  pattern \decm.  Again,
several families of pattern avoiding involutions  are equinumerous with
$\Sn_n^{(m)}$ (see~\cite{mbm-einar,dukes,jaggard-marincel,kratt-growth-diagrams}).
\begin{Theorem}\label{thm:i}
  The exponential \gf\ of \invs \ avoiding \decm\ is
$$
\sum_{\tau \in \In^{(m)} } \frac{t^{|\tau|}}{|\tau|!}=
  \left\{
  \begin{array}{ll}
  e^t\, \det \left(I_{i-j}- I_{i+j}\right)_{1\le i,j\le \ell},
& \hbox{ if } m=2\ell+1;
\\ 
\det \left(I_{i-j}+ I_{i+j-1}\right)_{1\le i,j\le \ell},
& \hbox{ if } m=2\ell,
 \end{array}\right.
$$
where $I_i$ is defined by~\eqref{Idef}.
\end{Theorem}
This result is obtained by applying Gessel's $\theta$ operator to a
Schur function identity due to Bender and Knuth~\cite{bender-knuth}.
The latter identity has been refined by taking into account the number of
columns of odd size in the tableaux (see Goulden~\cite{goulden}; 
Krattenthaler then gave a bijective proof of this refinement~\cite{krattenthaler-firenze}). Using the
operator $\theta$, and the properties of Schensted's
correspondence~\cite[Exercise~7.28]{stanley-vol2}, this translates into a refinement of
Theorem~\ref{thm:i} that takes into account  the number
$f(\tau)$ of fixed points in $\tau$. We shall also recover this result.
\begin{Theorem}\label{thm:i-f}
If $m=2\ell+1$,  the exponential \gf\ of \invs \ avoiding \decm\ and
having $p$ fixed points is
$$
\sum_{\tau \in \In^{(m)} , f(\tau)=p } \frac{t^{|\tau|}}{|\tau|!} =
  \frac {t^p} {p!}\, \det \left(I_{i-j}- I_{i+j}\right)_{1\le i,j\le \ell}.
$$
If $m=2\ell$,  this \gf\ is 
$$
\sum_{\tau \in \In^{(m)} \atop f(\tau)=p} \frac{t^{|\tau|}}{|\tau|!} 
= \det \left(
\begin{array}{c}
 (I_{p+\ell-j}-I_{p+\ell+j} )_{1\le j \le \ell}
\\
 (I_{i+j-1} - I_{i-j-1})_{2\le i \le \ell, 1\le j \le \ell},
\end{array}\right)
$$
where we have described separately the first row of the determinant and the
next $\ell-1$ rows ($i=2, \ldots, \ell)$.
\end{Theorem}
The first result of Theorem~\ref{thm:i-f} can be restated as follows: if $m=2\ell+1$,
the  \gf\ of \invs \ avoiding \decm, counted by the length and number of fixed
points is
\beq\label{inv-fp}
\sum_{\tau \in \In^{(m)}} \frac{t^{|\tau|}}{|\tau|!} s^{f(\tau)}=
   e^{st}\, \det \left(I_{i-j}- I_{i+j}\right)_{1\le i,j\le \ell}.
\eeq
It thus  appears as a very simple extension of the first part of
Theorem~\ref{thm:i}, and indeed, the connection between these two
formulas is easy to justify  combinatorially (the fixed points play no role when one forbids a decreasing
pattern of even length). 

\medskip
Let us now outline the structure of the paper. In Section~\ref{sec:gen},
we  describe how the ``catalytic'' parameters change in the
recursive construction  of \ps\ of
$\Sn^{(m)}$ and  $\In^{(m)}$. We do not give the proofs, as this was
done  by Guibert
and Jaggard \& Marincel,
respectively~\cite{guibert-these,jaggard-marincel}. We then convert
these descriptions into
the functional equations that are at the heart of this paper (Propositions~\ref{prop:eq-p} and~\ref{prop:eq-f-i}).
In Section~\ref{sec:ex}, we illustrate our approach by two simple
examples, namely the enumeration of \ps\ of 
$\Sn^{(2)}$ and of standard Young tableaux.
Next we return to \ps: we
first address in Section~\ref{sec:inv} the solution of the equation obtained
for involutions  of $\In^{(m)}$, and
finally, we solve in Section~\ref{sec:perm} the equation obtained for \ps\ of
$\Sn^{(m)}$. The reason why we address involutions first is that the
solution is  really elementary  in this
case. One step of the solution turns  out to be more difficult in the
case of \ps, although the basic ingredients are the same.

\medskip

Let us finish with some standard definitions and notation.
Let $A$ be a commutative ring and $x$ an indeterminate. We denote by
$A[x]$ 
(resp. $A[[x]]$) 
the ring of polynomials
 (resp. \fps) in $x$
with coefficients in $A$. If $A$ is a field, then $A(x)$ denotes the field
of rational functions in $x$ (with coefficients in $A$).
This notation is generalized to polynomials, fractions
and series in several indeterminates. We 
denote $\bx=1/x$, so that $A[x,\bx]$ is the ring of Laurent
polynomials in $x$ with coefficients in $A$.
A \emm Laurent series, is a series of the form $\sum_{n\ge n_0} a(n)
x^n$, for some $n_0\in \zs$.
The coefficient of $x^n$ in  $F(x)$ is denoted
$[x^n]F(x)$. 

Most of the series that we use in this paper are power series in $t$ with coefficients in
$A[x,\bx]$, that is, series of the form 
$$
F(x;t)=\sum_{n\ge 0, i\in \zs} f(i;n) x^i t^n,
$$
where for all $n$, almost all coefficients $f(i;n)$ are zero. The
\emm positive part, of $F(x;t)$ in $x$ is the following series, which
has coefficients in $xA[x]$:
$$
[x^>]F(x;t):=\sum_{n\ge 0, i>0} f(i;n) x^i t^n.
$$
We define similarly the negative, non-negative and non-positive parts of
$F(x;t)$ in $x$, which we denote respectively  by $[x^<]F(x;t), [x^\ge]F(x;t)$ and
$[x^\le]F(x;t)$.


\section{Generating trees and functional equations}
\label{sec:gen}
\subsection{Permutations avoiding \incm}
\label{sec:tree-p}
Take a  permutation $\pi$ of $\Sn_{n+1}^{(m)}$, written as the word
$\pi(1)\cdots \pi(n+1)$. Erase from this word the value
$n+1$: this gives a permutation $\tau$ of $\Sn_{n}^{(m)}$. This property allows us to display the \ps\ of $\Sn^{(m)}$ as the nodes of a
\emm generating tree,. At the root of this tree sits the unique \p\,
of length $0$, and the children of a node indexed by $\tau\in
\Sn_{n}^{(m)}$ are the \ps\ of $\Sn_{n+1}^{(m)}$ obtained by inserting
the value $n+1$ in $\tau$. In how many ways is this insertion possible?
If $\tau$ avoids $12\cdots m$, then all insertion positions are
\emm admissible,, that is, give a \p \ of  $\Sn_{n+1}^{(m)}$. There
are $n+1$ such positions. Otherwise, 
only the $a$ leftmost insertion positions are admissible, where $a$ is
the position of the leftmost occurrence of   $12\cdots m$ in 
$\tau$. More precisely:
$$
a= \min\,\{i_m: \exists\ i_1<i_2<\cdots <i_m \hbox{ s.t. } \tau(i_1)< \cdots
  < \tau(i_m)\}.
$$

As
we wish to describe \emm recursively, the shape of the
generating tree, we now need to find the position of  the leftmost
occurrence of   $12\cdots m$ in 
the children  of $\tau$. But it is easily seen that this depends
on the position of the leftmost occurrence of   $12\cdots (m-1)$ in
$\tau$. And so on! We are thus led to define the following $m$ parameters: for
$1\le j\le m$, and $\tau  \in \Sn_{n}^{(m)}$, let
\beq\label{aj-def}
a_j(\tau)=  
\left\{
\begin{array}{ll}
n+1, & \hbox{if } \tau \hbox{ avoids } 12\cdots j;
\\
  \min\,\{i_j: \exists\ i_1<i_2<\cdots <i_j \hbox{ s.t. } \tau(i_1)< \cdots
  < \tau(i_j)\},
& \hbox{otherwise.}
\end{array}
\right.
\eeq
Note that $a_1(\tau)=1$, and that $a_1(\tau)\le \cdots \le a_m(\tau)$. We call the sequence $L(\tau):=(a_2(\tau), \ldots,
a_m(\tau))$ the \emm label, of $\tau$. The empty \p\ has label $(1,
\ldots, 1)$. 

We can now describe the labels
of the children of $\tau$ in terms of $L(\tau)$ (Guibert~\cite[Prop.~4.47]{guibert-these}).

\begin{Proposition}
\label{prop:rec-p}
Let $\tau  \in \Sn_{n}^{(m)}$ with  $L(\tau)=(a_2, \ldots,
a_m)$. Denote $a_1=1$. The labels of the $a_m$ \ps\ of $\Sn_{n+1}^{(m)}$
obtained by inserting $n+1$ in $\tau$ are
$$
\left\{
\begin{array}{{ll}}
(a_2+1, a_3+1, \ldots, a_m+1)  & 
%
\\
(a_2, \ldots, a_{j-1}, \al , a_{j+1}+1, \ldots,  a_m+1)
 & \hbox{for } 2\le j \le m \hbox{ and } a_{j-1}+1 \le \al \le a_{j}.
\end{array}\right.
$$
\end{Proposition}
\noindent
The first label corresponds to an insertion in position 1, while the
label involving $\al$ corresponds to an insertion in position $\al$.
We refer the reader to Figure~\ref{fig:ex-1234} for an example.

\begin{figure}[ht!]
\begin{center} 
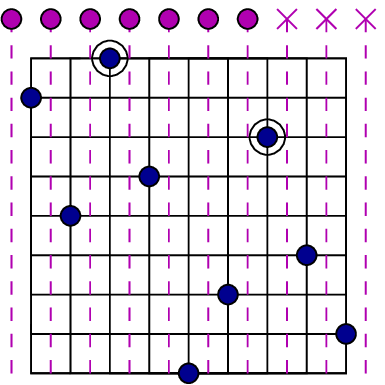
\caption{The \p\ $\tau = 8\ 5\ 9\ 6\ 1\ 3\  7\ 4\ 2 \ \ \in
  \Sn_9^{(3)}$. One has $a_1(\tau)=1$, $a_2(\tau)=3$,
  $a_3(\tau)=7$. There are 7 admissible ways to insert the value
  10. Inserting 10 to the right of  $\tau(7)$
  would create an occurrence of 1234.}
\label{fig:ex-1234} 
\end{center}
\end{figure}

\medskip
Let us now translate the recursive construction of \ps\ of $\Sn^{(m)}$
in terms of \gfs. Let $\Fa(u_2, \ldots,
u_m;t)$ be the (ordinary) \gf\ of \ps\  of $\Sn^{(m)}$, counted by
the statistics $a_2, \ldots, a_m$ and by the length:
\begin{eqnarray*}
  \Fa(u_2, \ldots,u_m;t)&=&\sum_{\tau \in \Sn^{(m)}} u_2^{a_2(\tau)}
\cdots u_m^{a_m(\tau)}t^{|\tau|}
\\
&=&
\sum_{a_2, \ldots, a_m} \Fa_{a_2, \ldots, a_m} (t) u_2^{a_2}\cdots u_m^{a_m}
\end{eqnarray*}
where $\Fa_{a_2, \ldots, a_m} (t)$ is the length \gf\ of \ps \ of
$\Sn^{(m)}$ having label $(a_2, \ldots, a_m)$.
We still denote $a_1=1$. The above proposition gives
\begin{multline*}
  \Fa(u_2, \ldots,u_m;t)= u_2\cdots u_m + 
tu_2\cdots u_m \Fa(u_2, \ldots,u_m;t)
\\
+t\, \sum_{a_2, \ldots, a_m} \Fa_{a_2, \ldots, a_m} (t) 
%
%
\sum_{j=2}^{m}
\sum_{\al=a_{j-1}+1}^{a_{j}}
 u_2^{a_2}\cdots u_{j-1}^{a_{j-1}} u_{j}^\al u_{j+1}^{a_{j+1}+1}
\cdots u_m^{a_m+1}.
\end{multline*}
Using
$$
\sum_{\al=a_{j-1}+1}^{a_{j}}u_{j}^\al=
\frac{u_{j}^{a_{j}+1}-u_{j}^{a_{j-1}+1} }
{u_j-1},
$$
we  obtain (given that $a_1=1$):
\begin{multline}\label{eq-func-p-u}
\Fa(u;t)= u_ {2,m}+  t u_ {2,m}\Fa(u;t)
+t  u_{2,m} \frac{\Fa(u;t)-u_2\Fa(1, u_3, \ldots,  u_m;t)}{u_2-1}
 \\
+ t\sum_{j=3}^{m}
u_{j,m} \frac{\Fa(u;t)-\Fa(u_2, \ldots, u_{j-2}, u_{j-1}u_j,1, u_{j+1},
    \ldots,  u_m;t)}{u_j-1}
\end{multline}
where $\Fa(u;t) \equiv \Fa(u_2, \ldots,u_m;t)$ and $u_{j,k}=u_j
u_{j+1} \cdots u_k$. 

To finish, let us perform an elementary transformation on the series $\Fa(u;t)$.
 Define
\beq\label{weight-p}
\Fc(v;t)=\Fc(v_1, \ldots, v_m;t)= \sum_{\tau\in \Sn^{(m)}}
v_1^{a_2-1} v_2^{a_3-a_2} \cdots v_m^{|\tau|+1-a_m} t^{|\tau|},
\eeq
where $(a_2, \ldots, a_m)=L(\tau)$. 
We have eliminated the dependence $a_2\le \cdots \le a_m$ between the
exponents of $u_2, \ldots, u_m$ in $\Fa(u;t)$. As will be seen below,
another effect of this change of series is that the cases $j=2$ and
$j=3, \ldots, m$ now play the same role. We also note that the
variable $t$ is now  redundant in $\Fc(v;t)$, but it is our main variable,
and we find it convenient to keep it.
The series $\Fa$ and $\Fc$ are  related by
$$
\Fc(v_1, \ldots, v_m;t)= \frac{v_m}{v_1}\Fa\left(\frac{v_1}{v_2},
\ldots, \frac{v_{m-1}}{v_m}; v_mt\right)
$$
and conversely
$$
\Fa(u_2, \ldots, u_m;v_m t)= u_{2,m} \Fc(u_{2,m}v_m, u_{3,m}v_m,
\ldots, u_m v_m, v_m; t)
$$
where as above $u_{j,k}=u_j
u_{j+1} \cdots u_k$. The functional equation~\eqref{eq-func-p-u}
satisfied by $\Fa(u;t)$ translates into an equation of  a slightly
simpler form satisfied by $\Fc(v;t)$.
\begin{Proposition}\label{prop:eq-p}
  The \gf\ $\Fc(v;t)\equiv \Fc(v_1, \ldots, v_m;t)$ of \ps\ of
  $\Sn^{(m)}$, defined by~\eqref{weight-p}, satisfies
$$
\Fc(v;t)= 1+ tv_1 \Fc(v;t)+ t\sum_{j=2}^{m} v_{j-1}v_{j}\ 
\frac{\Fc(v;t)-\Fc(v_1, \ldots, v_{j-2}, v_j, v_j, v_{j+1}, \ldots,
  v_m;t)}{v_{j-1}-v_j}.
$$
The series $\Fc(1, \ldots, 1;t)$ counts \ps\ of   $\Sn^{(m)}$ by their length.
\end{Proposition}
In Section~\ref{sec:perm}, we derive from this equation the Bessel
\gf\ of \ps\ of   $\Sn^{(m)}$, as given by Theorem~\ref{thm:p}.

\subsection{Involutions avoiding \decm}
\label{sec:tree-inv}
It follows from the properties 
of Schensted's correspondence~\cite{schensted}  that
the number of \invs\ of length $n$ avoiding \incm \ equals
the number of \invs\ of length $n$ avoiding \decm. 
However, this correspondence is not a simple symmetry, and the
generating trees  that describe \incm-avoiding \invs \ and
\decm-avoiding \invs\ are not isomorphic. Both trees are defined
by the same principle: the root is the empty \p\ and the parent of an
\inv\ $\pi$ is
obtained by deleting the cycle containing the largest entry, and
normalizing the resulting sequence. For
instance, if $\pi=426153$, the deletion of the 2-cycle $(3,6)$ first gives $4215$, and, after
normalization, $3214$. 

The tree that generates \incm-avoiding \invs\  is similar to the tree
generating \incm-avoiding \ps.  Its description
involves $m$ catalytic parameters ({Gui\-bert}~\cite[Prop.~4.52]{guibert-these}).
The tree that generates \decm-avoiding \invs\ requires 
$\lfloor m/2\rfloor$ catalytic parameters only
(Jaggard \& Marincel~\cite{jaggard-marincel}). The source of this
compactness is easy to understand:  an
\inv\ $\tau$ contains the pattern $k\cdots 21$ if and only 
if it contains a  \emm symmetric, occurrence of this pattern (by this,
we mean that the corresponding set of
points in the diagram of $\tau$  is symmetric with respect to the
first diagonal, see Figure~\ref{fig:ex-123456-i}).
Equivalently, this means that a decreasing subsequence of length
$\lceil k/2\rceil$ occurs 
in the points of the diagram  lying on or above the first
diagonal. Thus we only need to keep track of descending subsequences
of length at most $m/2$ (in the top part of the diagram), and we can
expect to have about $m/2$ catalytic parameters. 

\begin{figure}[ht!]
\begin{center} 
\input{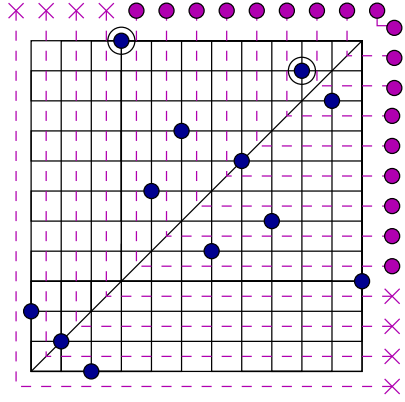}
\caption{The \inv\ $\tau = 3\ 2\ 1\ 12\  7\  9\  5\ 8\ 6\  11\  10\  4 \ \in
  \In_{12}^{(5)}$. One has $a_1(\tau)=3$, $a_2(\tau)=9$. There are $9$
  admissible ways to insert a $2$-cycle.}
\label{fig:ex-123456-i} 
\end{center}
\end{figure}

Let us now describe in details the tree generating \decm-avoiding \invs. The
example of Figure~\ref{fig:ex-123456-i} illustrates the argument. 
Let $\tau$ be an \inv\ of $\In_n^{(m)}$. Inserting $n+1$ as a fixed
point in $\tau$ always gives an involution of  $\In^{(m)}$. For
$1\le i\le n+1$, let us now
consider the \p\ $\pi$ obtained by adding 1 to all values larger than
or equal to $i$, and inserting the 2-cycle
$(i , n+2)$. How
many of these insertions are admissible, that is, give an involution
of   $\In^{(m)}$? If $\tau$ avoids $(m-1)\cdots 21$, then all
insertions are admissible, including the most ``risky'' one,
corresponding to $i=1$. Otherwise,  the only admissible values of $i$
are $n+1, n, \ldots, n-a+2$, where $n-a+1$ is the 
position of the rightmost  symmetric occurrence of $(m-1)\cdots
21$. In other words, if we denote $m=2\ell +\eps$ with $\eps\in\{0,1\}$,
$$
n-a+1=\max\,\{ i_1: \exists\ i_1<i_2<\cdots <i_{\ell} \hbox{ s.t. }
\tau(i_1) >\cdots >\tau(i_\ell)\ge i_\ell+\eps\}.
$$
Again, in order to keep track  of this parameter recursively, we are
led to define, for $1\le j \le \ell$, the following $\ell$ catalytic
parameters: 
$$
a_j(\tau)=  
\left\{
\begin{array}{ll}
n+1, \hskip 76mm \hbox{if } \tau \hbox{ avoids } (2j-1+\eps)\cdots 21;
\\
%
n+1  -\max\,\{i_1: \exists\ i_1<i_2<\cdots <i_{j} \hbox{ s.t. }
\tau(i_1)> \cdots  > \tau(i_j)\ge i_j+\eps\},
\hskip 3mm \hbox{otherwise.}
\end{array}
\right.
$$
In particular, $a_\ell(\tau)$ is the parameter that was denoted $a$
above, and it is also the number of admissible insertions of a 2-cycle
in $\tau$. We call the sequence $L(\tau):=(a_1(\tau), \ldots,
a_\ell(\tau))$ the \emm label, of $\tau$. Note that $a_1(\tau)\le \cdots
\le a_\ell(\tau)$. The empty \p\ has label $(1,
\ldots, 1)$. 

We can now describe the labels
of the children of $\tau$ in terms of $L(\tau)$.

\begin{Proposition}[Jaggard \& Marincel~\cite{jaggard-marincel}]
\label{prop:inv-tree}
Let $\tau $ be an \inv\ in $\In^{(m)}$ with  $L(\tau)=(a_1, \ldots,
a_\ell)$. Denote $a_0=0$. The labels of the $a_\ell$ \invs\ of $\In^{(m)}$ obtained by inserting a cycle in $\tau$ are
$$
\left\{
\begin{array}{{ll}}
(a_1+1, a_2+1, \ldots, a_\ell+1),  & \hbox{ if } m \hbox { is odd; }
\\
(1, a_2+1, \ldots, a_\ell+1),  & \hbox{ if } m \hbox { is even; }
\\
(a_1+1, \ldots, a_{j-1}+1, \al , a_{j+1}+2, \ldots,  a_\ell+2)
 & \hbox{ for } 1\le j \le \ell \hbox{ and } a_{j-1}+2 \le \al \le a_{j}+1.
\end{array}\right.
$$
The first two labels correspond to the insertion of a fixed point, the
other ones to the insertion of a $2$-cycle.
\end{Proposition}
\noindent
We refer again the reader to Figure~\ref{fig:ex-123456-i} for an
example.

\medskip
Let us now translate the recursive construction of \invs\ of $\In^{(m)}$
 in terms of \gfs. Let $\Ga(u_1, \ldots,
u_\ell;t)$ be the (ordinary) \gf\ of \invs\  of $\In^{(m)}$, counted
by the statistics $a_1, \ldots, a_\ell$ and by the length:
\begin{eqnarray*}
  \Ga(u_1, \ldots,u_\ell;t)&=&\sum_{\tau \in \In^{(m)}} u_1^{a_1(\tau)}
\cdots u_\ell^{a_\ell(\tau)}t^{|\tau|}
\\
&=&
\sum_{a_1, \ldots, a_\ell} \Ga_{a_1, \ldots, a_\ell} (t) u_1^{a_1}\cdots u_\ell^{a_\ell}
\end{eqnarray*}
where $\Ga_{a_1, \ldots, a_\ell} (t)$ is the length \gf\ of \ps \ of
$\In^{(m)}$ having label $(a_1, \ldots, a_\ell)$.
We still denote $a_0=0$. The above proposition gives
\begin{multline*}
  \Ga(u_1, \ldots,u_\ell;t)= u_1\cdots u_\ell + 
tu_1\cdots u_\ell \Ga(u_1, \ldots,u_\ell;t)\chi_{m\equiv {1}}
+tu_1\cdots u_\ell \Ga(1,u_2, \ldots,u_\ell;t)\chi_{m\equiv 0}
\\
+t^2\, \sum_{a_1, \ldots, a_\ell} \Ga_{a_1, \ldots, a_\ell} (t) 
%
%
\sum_{j=1}^{\ell}
\sum_{\al=a_{j-1}+2}^{a_{j}+1}
 u_1^{a_1+1}\cdots u_{j-1}^{a_{j-1}+1} u_{j}^\al u_{j+1}^{a_{j+1}+2}
\cdots u_\ell^{a_\ell+2},
\end{multline*}
where $\chi_{m\equiv {i}}$ equals $1$ if $m$  equals $i$ modulo 2, and 0
otherwise.
Using
$$
\sum_{\al=a_{j-1}+2}^{a_{j}+1}u_{j}^\al=
\frac{u_{j}^{a_{j}+2}-u_{j}^{a_{j-1}+2} }
{u_j-1},
$$
we finally obtain (given that $a_0=0$):
\begin{multline}\label{eq-func-i-u}
\Ga(u;t)= u_ {1,\ell}+  t u_ {1,\ell}\Ga(u;t)\chi_{m\equiv {1}}
+  t u_ {1,\ell}\Ga(1,u_2, \ldots,u_\ell;t)\chi_{m\equiv {0}}
 \\
+ t^2u_ {1,\ell}\sum_{j=1}^{\ell}
u_{j,\ell} \frac{\Ga(u;t)-\Ga(u_1, \ldots, u_{j-2}, u_{j-1}u_j,1, u_{j+1},
    \ldots,  u_\ell;t)}{u_j-1}
\end{multline}
where $\Ga(u;t) \equiv \Ga(u_1, \ldots,u_\ell;t)$ and $u_{j,k}=u_j
u_{j+1} \cdots u_k$. 

To finish, let us perform an elementary transformation on the series $\Ga(u;t)$.
 Define
\beq\label{weight-i}
\Gc(v;t)=\Gc(v_1, \ldots, v_\ell;t)= \sum_{\tau\in \In^{(m)}}
v_1^{a_1} v_2^{a_2-a_1} \cdots v_\ell^{a_\ell-a_{\ell-1}} t^{|\tau|},
\eeq
where $(a_1, \ldots, a_\ell)=\ell(\tau)$.
We have eliminated the dependence $a_1\le \cdots \le a_\ell$ between the
exponents of $u_1, \ldots, u_\ell$ in $\Ga(u;t)$. The series $\Ga$ and $\Gc$ are  related by
$$
\Gc(v_1, \ldots, v_\ell;t)= 
\Ga\left( \frac{v_1}{v_2},\ldots, \frac{v_{\ell-1}}{v_\ell},v_\ell;t\right),
$$
and conversely
$$
\Ga(u_1, \ldots, u_\ell; t)=  \Gc(u_{1,\ell}, u_{2,\ell},
\ldots, u_\ell; t)
$$
where as above $u_{j,k}=u_j
u_{j+1} \cdots u_k$. The functional equation~\eqref{eq-func-i-u}
satisfied by $\Ga(u;t)$ translates as follows.
\begin{Proposition}\label{prop:eq-f-i}
  The \gf\ $\Gc(v;t)\equiv \Gc(v_1, \ldots, v_\ell;t)$ of \invs\ of
  $\In^{(m)}$, defined by~\eqref{weight-i}, satisfies
  \begin{multline*}
\Gc(v;t)= v_1+ tv_1 \Gc(v;t)\chi_{m\equiv {1}}
+ tv_1 \Gc(v_2,v_2,v_3, \ldots, v_\ell;t) \chi_{m\equiv {0}}
\\
+t^2 v_1 \sum_{j=1}^{\ell} v_{j}v_{j+1}\ 
\frac{\Gc(v;t)-\Gc(v_1, \ldots, v_{j-1}, v_{j+1}, v_{j+1}, v_{j+2}, \ldots,
  v_\ell;t)}{v_{j}-v_{j+1}}.
\end{multline*}
The series $\Gc(1, \ldots, 1;t)$ counts \invs\ of   $\In^{(m)}$ by their length.
\end{Proposition}
In Section~\ref{sec:inv}, we derive from this equation the exponential
\gf\ of \invs\ of   $\In^{(m)}$, as given by Theorem~\ref{thm:i}.
We then refine the result to take into account the number of fixed points.

\section{Two examples}
\label{sec:ex}
In this section, we illustrate the ingredients of our solution of the
equations of Propositions~\ref{prop:eq-p} and~\ref{prop:eq-f-i} by
taking two examples. The first one deals with the enumeration of
123-avoiding \ps. The second one is a \gf\ proof of MacMahon's formula
for the number of standard tableaux of a given shape,
and should clarify what we meant in the introduction by ``the reflection
principle performed at the level of power series''.

\subsection{Permutations  avoiding $123$}
In the introduction, we wrote the following equation for the
bivariate \gf\ of 123-avoiding \ps, counted by the position of the
first ascent and the length:
$$
\left(1-t\, \frac{u^2}{u-1}\right) F(u;t)= 1-t\, \frac {u}{u-1} F(1;t).
$$
This is the case $m=2$ of Proposition~\ref{prop:eq-p}, with $v_1=u$
and $v_2=1$. 

As explained in Section~1, this equation can be solved by an
appropriate choice of $u$ that cancels the kernel, and thus eliminates the
unknown series $F(u;t)$. This is the
standard kernel method. We present here
an alternative solution, sometimes called the \emm algebraic, kernel
method~\cite{Bous05,mbm-mishna}, where instead $F(1;t)$ is
eliminated. This elimination is obtained by 
exploiting a certain symmetry of the kernel.
This symmetry appears clearly if we set $u=1+x$. The equation then reads:
$$
\left( 1-t\, {(1+x)(1+\bx)}  \right) F(1+x;t)= 1-t(1+\bx) F(1;t)
$$
with $\bx=1/x$. {The kernel is now invariant under $x\mapsto \bx$.}
Replace $x$ by $\bx$:
$$
\left( 1-t\, {(1+x)(1+\bx)}\right) F(1+\bx;t)= 1-t(1+x) F(1;t).
$$
{We  now eliminate $F(1;t)$} by  taking a linear combination of these two
equations. This leaves:
\beq\label{os-simple}
\left( 1-t\, {(1+x)(1+\bx)}\right) \left( F(1+x;t)-\bx F(1+\bx;t)\right)
= 1-\bx,
\eeq
or
$$
{F(1+x;t)-\bx F(1+\bx;t)= \frac{ 1-\bx}{ 1-t {(1+x)(1+\bx)} }}:=R(x;t).
$$
In this equation,
\begin{itemize}
\item[--] $F(1+x;t)$ is a series in $t$ with coefficients in $\qs[x]$,
\item[--] $\bx F(1+\bx;t)$ is a series in $t$ with coefficients in $\bx \qs[\bx]$,
\item[--]  the right-hand side  $R(x;t)$ is a series in $t$ with
  coefficients in $ \qs[x,\bx]$.
\end{itemize}
Consequently, $F(1+x;t)$ is the non-negative part of $R(x;t)$ in  $x$.
In particular, the length \gf\ of 123-avoiding \ps\ is
\begin{eqnarray}
F(1;t)= [x^0] R(x;t)&=& \sum_{n\ge 0} [x^0] ( 1-\bx)\bx^n (1+x)^{2n}
t^n
\label{extr}\\ 
&=& \sum_{n\ge 0} \left( {{2n}\choose n} -{{2n}\choose {n+1}}\right)
t^n \nonumber\\
&=& \sum_{n\ge 0} \frac{t^n}{n+1}  {{2n}\choose n}.\nonumber
\end{eqnarray}
This small example contains all ingredients of what will be our solution for
a generic value of $m$: 
\begin{itemize}
\item[--] a change of variables, which may not have a clear combinatorial
  meaning,
\item[--] a finite group $G$ acting on power series that leaves the kernel
  unchanged (here, the group has order 2, and replaces $x$ by $1/x$),
\item[--] a linear combination~\eqref{os-simple} of all the equations obtained by letting an
  element of $G$ act on the original functional equation; in this
  linear combination, called the \emm orbit sum,, the left-hand side
  is a multiple of the kernel, and the right-hand side does not
  contain any unknown series,
\item[--] finally, a coefficient extraction~\eqref{extr} that gives the \gf\ under interest.
\end{itemize}
Let us mention, however, that for a generic value of $m$, the change of variables used in
Section~\ref{sec:perm} is not a direct extension  of $v\mapsto
1+x$. But, on this small example, the latter choice is simpler.

\subsection{Standard Young tableaux}
Let $\la=(\la_1, \ldots, \la_m)\in \ns^m$ be an integer partition. That is,
$\la_1\ge \cdots \ge \la_m\ge 0$. The \emm weight, of $\la$ is
$|\la|:=\la_1+ \cdots+ \la_m$.
We identify $\la$ with its Ferrers shape, in
which the $i^{\hbox{\small th}}$ row has $\la_i$ cells. A \emm standard
tableau, of shape $\la$ is a filling of the cells of $\la$ with the
integers $1, 2, \ldots, |\la|$, that increases along rows and
columns (Figure~\ref{fig:young}). The \emm height, of the tableau is  the number of non-empty rows, that is $\max(i : \la_i>0)$. 
Let $f^\la$ denote the number of standard Young
tableaux of shape $\la$.

\begin{figure}[ht!]
\begin{center} 
{\includegraphics{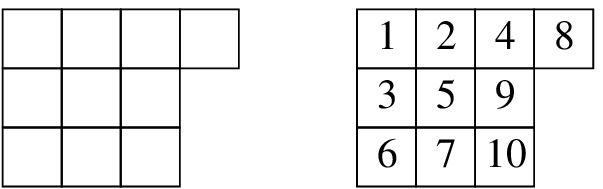}}
\caption{The Ferrers shape associated with the partition $\la=(4,3,3)$
  and a standard tableau of shape $\la$.}
\label{fig:young}
\end{center}
\end{figure}

Our objective here is to recover the hook-length formula, or, rather,
an equivalent form due to MacMahon~\cite[Sec.~III, Chap.~V]{macmahon}.
\begin{Proposition}
  Let $\la=(\la_1, \ldots, \la_m)$ be a partition of weight $n$. 
 The number of standard Young tableaux of shape $\la$ is
$$
f^\la = \frac{n!}{\prod_{i=1}^m (\la_i-i+m)!}\ \prod_{1\le i<j\le m}
(\la_i-\la_j-i+j).
$$
\end{Proposition}
\begin{proof}
  Let $F(u)\equiv F(u_1, \ldots, u_m)$ be the \gf\ of standard
  tableaux of height at most $m$:
$$
F(u):= \sum_{\la_1\ge  \cdots\ge  \la_m\ge 0} f ^{\la} \prod_{i=1}^m
u_i^{\la_i}.
$$
For $j=2, \ldots, m$,  we denote by  $F_j(u_1, \ldots,
u_{j-2}, u_{j-1}u_j, u_{j+1}, \ldots, u_m)\equiv F_j(u)$ the \gf\ of standard
tableaux such that the parts $\la_{j-1}$ and $\la_{j}$ are equal. This series is
obtained by 
extracting the corresponding terms from $F(u)$ (it is also called the
$(j-1,j)$-diagonal of $F(u)$). In all terms of this series, $u_{j-1}$
and $u_j$ appear with the same exponent, which allows us to
write this series in the above form.

Now a tableau of weight $n+1$ is obtained by adding a cell labelled
$n+1$ to a tableau of weight $n$. This cell can be added to the $j^{\hbox{\small
    th}}$  row
   unless this row should have the same length as the $(j-1)^{\hbox{\small
    st}}$ row. This gives directly
the following equation:
$$
F(u)=1+ u_1 F(u)+ \sum_{j=2}^m u_j\left( F(u)-F_j(u)\right),
$$
that is,
$$
\Big(1- \sum_{j=1}^m u_j\Big) F(u)
= 1- \sum_{j=2}^m u_jF_j(u).
$$
Observe that the kernel $K(u):= 1-\sum u_j$ is  invariant under the
action of the symmetric group
$\Sn_m$, seen as a group of transformations of polynomials in $u_1,
\ldots, u_m$. This group is
generated by $m-1$ elements of order 2, denoted $\si_1, \ldots,
\si_{m-1}$:
$$
\si_j(P(u_1,  \ldots, u_m))= P(u_1,\ldots, u_{j-1},  u_{j+1},u_j,
u_{j+2}, \ldots, u_m).
$$
Let us multiply the equation by $M(u):= u_1^{m-1} \cdots u_{m-1}^1
u_m^0$. This gives:
\beq\label{young-eq}
K(u) M(u) F(u)=M(u)
-\sum_{j=2}^m u_1^{m-1} \cdots u_{j-1}^{m-(j-1)}u_j^{m-j+1}  \cdots
u_m^0 F_j(u).
\eeq
Recall that $F_j(u)$ stands for $F_j(u_1, \ldots,
u_{j-2}, u_{j-1}u_j, u_{j+1}, \ldots, u_m)$. Hence the $j^{\hbox{\small
    th}}$ term in the above sum is invariant under the action of the generator $\si_{j-1}$ (which
exchanges $u_{j-1}$ and $u_j$). Consequently, forming the signed sum
of~\eqref{young-eq} over the symmetric group $\Sn_m$ gives the
following \emm orbit sum,, which does not involve the series $F_j$:
$$
\sum_{\si \in \Sn_m} \vareps(\si)\ \si\!\left(K(u) M(u) F(u)\right)
=\sum_{\si \in \Sn_m} \vareps(\si)\ \si\!\left( M(u)\right),
$$
or, given that $K(u)$ is $\Sn_m$-invariant,
\beq\label{os-young}
\sum_{\si \in \Sn_m} \vareps(\si)\ \si\!\left(M(u) F(u)\right)
=\frac{\sum_{\si \in \Sn_m} \vareps(\si)\ \si\!\left( M(u)\right)}{K(u)}.
\eeq
Of course, the sum on the right-hand side can be evaluated explicitly
(the numerator is the Vandermonde determinant), but this will not be
needed here.

\smallskip
We claim that the number $f^\la$ can be simply obtained by a
coefficient extraction in the above identity. Consider the series
$M(u)F(u)$. Each monomial $u_1^{a_1} \cdots u_m^{a_m}$ that occurs in
it satisfies $a_1 >\cdots >a_m$ (because $a_i=m-i+\la_i$, where $\la$
is a partition). Consequently, if $\si$ is not the identity, the
exponents of any  monomial $u_1^{a_1} \cdots u_m^{a_m}$  occurring in
$\si\left(M(u) F(u)\right)$ are totally ordered \emm in a different
way,. Hence, when we extract the coefficient of  $u_1^{m-1+\la_1}
\cdots u_m^{0+\la_m}$ from~\eqref{os-young}, only the term
corresponding to $\si=\id$ contributes in the left-hand side, so that
$$
f^{\la}= [u_1^{m-1+\la_1}\cdots u_m^{0+\la_m}]\ \frac{\sum_{\si \in
      \Sn_m} \vareps(\si)\ \si\left( M(u)\right)}{K(u)}.
$$
Given that  $M(u)= u_1^{m-1} \cdots u_{m-1}^1
u_m^0$ and 
$$
\frac 1{K(u)}=\frac 1 {1-\sum_{j=1}^m u_j}=\sum_{a_1, \ldots, a_m \ge
  0} \frac{(a_1+\cdots + a_m)!}{\prod_{i=1}^m a_i!}\ u_1^{a_1}\cdots u_m^{a_m},
$$
we obtain
\begin{eqnarray*}
f^{\la}&=& \sum_{\si \in    \Sn_m} \vareps(\si)\frac{(\la_1+\cdots +
  \la_m)!}
{\prod_{i=1}^m (\la_i-i+\si^{-1}(i))!}\\
&=&
n! \det \left( \frac 1  {(\la_i-i+j)!}\right)_{1\le i, j \le m}\\
&=& n! \det  \left( \frac {(\la_i-i+j+1) \cdots (\la_i-i+m)}
{(\la_i-i+m)!}\right)_{1\le i,j\le m}
\\
&=& \frac {n!}{ \prod_{i=1}^m {(\la_i-i+m)!}}\ 
  \det  \left(  (\la_i-i+j+1) \cdots (\la_i-i+m) \right)_{1\le i,j\le      m}
.
\end{eqnarray*}
The  $(i,j)$-coefficient of the latter  determinant is a polynomial in
$\la_i-i$ of degree  $m-j$ and leading coefficient $1$. Hence the
determinant is simply the Vandermonde determinant $\det( (\la_i-i)^{m-j})$,
  that is, $\prod_{i<j} (\la_i-\la_j-i+j)$. This completes the proof
  of the proposition.
\end{proof}
We recognize in this proof three of the four ingredients that were
used in the enumeration of 123-avoiding \ps: the finite group that
leaves the kernel invariant (here, $\Sn_m$), the orbit
sum~\eqref{os-young}, and the final coefficient 
extraction. In this example, the symmetries of the kernel are obvious
already with the original variables $u_i$, so that no change of variables is
required.

 This proof is the \gf\ counterpart
of the classical proof that encodes tableaux of height at most $m$  by
paths in $\ns^m$ formed of unit positive steps, that start from $(0,
\ldots, 0)$ and remain in the wedge $x_1\ge \cdots \ge x_m\ge 0$,
and then uses the reflection principle. It is also very close to
another proof due to Xin~\cite[Sec.~3.1]{xin}.

\section{Involutions with no long descending subsequence}
\label{sec:inv}

We now address the solution of the functional equation of
Proposition~\ref{prop:eq-f-i}, which defines the \gf\ of \invs\ avoiding
\decm. 

\subsection{Invariance properties of the kernel}
\label{sec:inv-i}
As discussed in the previous section, our objective is to
exploit invariance properties of the \emm kernel,, that is, the coefficient of
$\Gc(v;t)$. Let us first divide the equation of
Proposition~\ref{prop:eq-f-i} by $v_1$. Then the kernel reads
$$
\frac 1 {v_1}-t\chi_{m\equiv 1}-t^2\sum_{j=1}^\ell\frac{v_jv_{j+1}}{v_j-v_{j+1}}.
$$
The invariance properties of this rational function appear clearly
after performing the following change of variables:
\beq\label{vx}
v_i= \frac 1 {1-t(x_i+\cdots+ x_\ell)}.
\eeq
Indeed,  the kernel becomes
$$
K(x;t)=1-t(x_1+\cdots+ x_\ell)-t\chi_{m\equiv 1}-t(\bx_1+\cdots+
\bx_\ell),
$$
where $\bar x_i=1/x_i$, and is  invariant under the action of 
the hyperoctahedral group
$B_\ell$ (the group of signed permutations), seen as a group of
transformations on Laurent polynomials in  $x_1, \ldots, x_\ell$. This
group is generated by $\ell$ elements of order 2, denoted $\si_1,
\ldots, \si_\ell$:
$$
\si _j( P(x_1, \ldots, x_\ell)) = 
\left\{
\begin{array}{ll}
   P(\bx_1, x_2, \ldots, x_\ell), & \hbox{ if } j=1;
\\
P(x_1, \ldots, x_{j-2}, x_j ,x_{j-1},
x_{j+1}, \ldots, x_\ell),
&\hbox{ for } j\ge 2.
\end{array}\right.
$$
The equation of Proposition~\ref{prop:eq-f-i} now reads:
$$
K(x;t) \Gd(x;t)=1+ t \Gd(0,x_2, \ldots, x_\ell)\chi_{m\equiv 0}
 - t \sum_{j=1}^\ell
\bx_j \Gd(x_1, \ldots, x_{j-2},x_{j-1}+x_j, 0, x_{j+1}, \ldots,
x_\ell)
$$
where
$$
\Gd(x;t)\equiv \Gd(x_1, \ldots, x_\ell;t)=\Gc\left(\frac 1{1-t(x_1+\cdots+
  x_\ell)},  \frac 1{1-t(x_2+\cdots+ x_\ell)},\ldots, 
\frac 1{1-tx_\ell};t\right).
$$
\subsection{Orbit sum}
\label{sec:orbit-i}
\noindent We now handle separately the odd and even case.

\noindent $\bullet$ If $m$ is odd,
the equation reads 
$$
  K(x;t) \Gd(x;t)=1 - t \sum_{j=1}^\ell
\bx_j \Gd(x_1, \ldots, x_{j-2},x_{j-1}+x_j, 0, x_{j+1}, \ldots,
x_\ell;t)
$$
where
\beq\label{K-odd}
K(x;t)=1-t(1+x_1+\cdots+ x_\ell+\bx_1+\cdots+\bx_\ell).
\eeq
Let us multiply the equation by 
\beq\label{M-odd}
M(x):=x_1 x_2^2\cdots x_\ell^\ell.
\eeq
 This gives:
\begin{multline}\label{Gd-sym}
K(x;t) M(x)  \Gd(x;t)= M(x)\\
-  t \sum_{j=1}^\ell
x_1 \cdots x_ {j-1}^{j-1} x_j^{j-1} x_{j+1}^{j+1} \cdots x_\ell^\ell \Gd(x_1, \ldots, x_{j-2},x_{j-1}+x_j, 0, x_{j+1}, \ldots,
x_\ell;t).
\end{multline}
The first term ($j=1$) of the sum reads $x_2^2 \cdots x_\ell^\ell
\Gd(0,x_2\ldots, x_\ell)$ and is invariant under the action of the generator $\si_1$ of
$\B_\ell$ (which replaces $x_1$ by $\bx_1$).
For $j\ge 2$, the $j^{\hbox{\small th}}$ term of the sum is invariant
under the action of the
generator $\si_j$ (which exchanges $x_{j-1}$ and $x_j$).
Consequently, forming the signed sum of~\eqref{Gd-sym} over the
hyperoctahedral group $\B_\ell$ gives the following orbit sum:
$$
\sum_{\si \in B_\ell} \vareps(\si)\  \si \left( K(x;t) M(x)
\Gd(x;t)\right)
=\sum_{\si \in B_\ell} \vareps(\si) \si ( M(x)),
$$
or, given that $K(x;t)$ is $B_\ell$-invariant,
\beq\label{os}
\sum_{\si \in B_\ell} \vareps(\si)\  \si \left(  M(x)
\Gd(x;t)\right)
=\frac{\sum_{\si \in B_\ell} \vareps(\si)\ \si (M(x))}{K(x;t)},
\eeq
where $K(x;t)$ is given by~\eqref{K-odd} and $M(x)$
by~\eqref{M-odd}.

\medskip
\noindent $\bullet$ If $m$ is even,
the equation reads 
$$
  K(x;t) \Gd(x;t)=1 + t(1-\bx_1) \Gd(0,x_2, \ldots, x_\ell;t)
- t \sum_{j=2}^\ell\bx_j \Gd(x_1, \ldots, x_{j-2},x_{j-1}+x_j, 0, x_{j+1}, \ldots,
x_\ell;t)
$$
where
\beq\label{K-even}
K(x;t)=1-t(x_1+\cdots+ x_\ell+\bx_1+\cdots+\bx_\ell).
\eeq
Let us multiply the equation by 
\beq\label{M-even}
M(x):= x_2x_3^2\cdots x_\ell^{\ell-1}(1-x_1)\cdots(1-x_\ell).
\eeq
 This gives:
\begin{multline}\label{Gd-sym-even}
K(x;t) M(x)  \Gd(x;t)= M(x)+ tx_2x_3^2\cdots x_\ell^{\ell-1}
(1-\bx_1)(1-x_1)\prod_{j=2}^\ell(1-x_j) \Gd(0,x_2, \ldots, x_\ell;t)\\
-  t \prod_{j=1}^\ell (1-x_j)\sum_{j=2}^\ell
x_2 \cdots x_ {j-1}^{j-2} x_j^{j-2} x_{j+1}^{j} \cdots x_\ell^{\ell-1} \Gd(x_1, \ldots, x_{j-2},x_{j-1}+x_j, 0, x_{j+1}, \ldots,
x_\ell;t).
\end{multline}
The  term involving $\Gd(0,x_2, \ldots, x_\ell)$ is invariant under
the action of the generator $\si_1$ of
$\B_\ell$.
For $j\ge 2$, the $j^{\hbox{\small th}}$ term of the sum is invariant
under  the action of the
generator $\si_j$.
Consequently, forming the signed sum of~\eqref{Gd-sym-even} over the
hyperoctahedral group $\B_\ell$ yields the  orbit sum~\eqref{os}, where now
$K(x;t)$ and $M(x)$ are respectively given by~\eqref{K-even}
and~\eqref{M-even}. 

\subsection{Extraction of $\Gc(1,\ldots, 1;t)$}
\label{sec:extr-i}

\noindent $\bullet$ Assume $m$ is odd, and consider the orbit
sum~\eqref{os}. For every $\si \in B_\ell$, the term 
$$
\si(M(x) \Gd(x;t))= \si\left(x_1x_2^2\cdots x_\ell^\ell\  \Gc\left(\frac 1{1-t(x_1+\cdots+
  x_\ell)},  \frac 1{1-t(x_2+\cdots+ x_\ell)},\ldots, 
\frac 1{1-tx_\ell};t\right)\right)
$$
is a power series in $t$ with coefficients in $\qs[x_1, \ldots,
  x_\ell, \bx_1, \ldots, \bx_\ell]$. We will prove that the coefficient
of $x_1\cdots x_\ell^\ell$ in~\eqref{os} reduces to $\Gd(0,
\ldots,0;t)=\Gc(1, \ldots, 1;t)$, which is the (ordinary) length
\gf\ of \invs\ avoiding \decm.

First, if $\si$ has some  signed elements, all monomials in the $x_i$'s occurring in
$\si(M(x) \Gd(x;t))$ have at least one negative exponent. Hence
$\si(M(x) \Gd(x;t))$  does not contribute to the coefficient of
$x_1\cdots x_\ell^\ell$.

If $\si$ is not signed, it is a mere permutation
of the $x_i$'s. 
Each monomial occurring in $\si(M(x)
\Gd(x;t))$ is of the form $x_1^{e_1}\cdots x_\ell^{e_\ell}$, where the
$e_i$'s are positive. However, monomials with $e_1=1$ only occur if
$\si(1)=1$ (because of the factor $M(x)=x_1x_2^2\cdots x_\ell^\ell$). 
But then, if we also want $e_2=2$, the only \ps\ $\si$ that contribute
are those that satisfy $\si(2)=2$. Iterating this observation, we see
that the only \p\ $\si$ that contributes to the coefficient of
$x_1x_2^2\cdots x_\ell^\ell$ is the identity. Moreover, its
contribution is clearly $\Gd(0,\ldots, 0;t)=\Gc(1, \ldots, 1;t)$.

Let us state this as a proposition, in which we have also made explicit
the right-hand side of the orbit sum. 

\begin{Proposition}\label{prop:inv-m-odd}
If $m=2\ell+1$,  the ordinary \gf\ of \invs\ avoiding \decm\ is
the coefficient of $x_1x_2^2\cdots x_\ell^\ell$ in a rational function:
 $$
\Ge_m(t):=\sum_{\tau \in \In^{(m)} } t^{|\tau|}
=
\left[x_1x_2^2\cdots x_\ell^\ell\right]
\frac{\det(x_j^i-\bx_j^i)_{1\le i, j \le \ell}}{1-t(1+x_1+\cdots +
  x_\ell+ \bx_1+\cdots +\bx_\ell)}
.
$$
Equivalently, the \emm exponential, \gf\ of these involutions is
$$
\Ge_m^{(e)}(t):=\sum_{\tau \in \In^{(m)} } \frac{t^{|\tau|}}{|\tau|!}
=
e^t\, \left[x_1x_2^2\cdots x_\ell^\ell\right]
\det\left((x_j^i-\bx_j^i)e^{t(x_j+\bx_j)}\right)_{1\le i, j \le \ell}.
$$
\end{Proposition}
\begin{proof}
  We have just argued that $\Ge_m(t)$ is the coefficient of
  $x_1x_2^2\cdots x_\ell^\ell$ in the right-hand side of~\eqref{os}.
There remains to prove that
$$
\sum_{\si\in B_\ell} \vareps(\si) \si(x_1^1\cdots x_\ell^\ell)
= \det(x_j^i-\bx_j^i)_{1\le i, j \le \ell}.
$$
 This is easily proved if we consider that $\si$ first
replaces some $x_i$'s by their reciprocals, and then permutes the $x_i$'s. More
precisely, there is a bijection between $B_\ell$ and $\Sn_\ell\times
\zs_2^\ell$, sending $\si$ to $(\pi, e_1, \ldots, e_\ell)$, with
$\pi\in \Sn_\ell$ and $e_i\in \{-1, 1\}$, such that
\beq\label{Bl-Sl}
\si (P(x_1, \ldots, x_\ell))= \pi \left(P(x_1^{e_1}, \ldots,
x_\ell^{e_\ell})\right)
\quad \hbox{ and } \quad \vareps(\si)= \vareps(\pi)(-1)^{\sharp\{i:\,
  e_i=-1 \}}.
\eeq
Thus
\begin{eqnarray*}
  \sum_{\si\in B_\ell} \vareps(\si) \si(x_1^1\cdots x_\ell^\ell)
&=& 
\sum_{\pi\in \Sn_\ell} \vareps(\pi)\sum_{e_1, \ldots, e_\ell \in
  \{-1,1\}}
(-1)^{\sharp\{i:\,  e_i=-1 \}} x_{\pi(1)}^{e_1}x_{\pi(2)}^{2e_2}\cdots
x_{\pi(\ell)}^{\ell e_\ell}
\\
&=&
\sum_{\pi\in \Sn_\ell} \vareps(\pi) \prod_{i=1}^\ell 
\left(x_{\pi(i)}^{i}-\bx_{\pi(i)}^{i}\right)
\\
&=&
\det(x_j^i-\bx_j^i)_{1\le i, j \le \ell}.
\end{eqnarray*}
This gives the expression of $\Ge_m(t)$. We then convert it into an
expression for the exponential \gf\ $\Ge_m^{(e)}(t)$ by observing that
the ordinary \gf\ $1/(1-a t)=\sum_n t^n a^n$ 
corresponds to the exponential \gf\ $\sum_n t^n a^n/n!=\exp(at)$. 
\end{proof}

\noindent{\bf Remark.} The determinant occurring in the proposition
can be evaluated in closed form~(see, e.g., \cite[Lemma~2]{kratt-det}):
$$
\det(x_j^i-\bx_j^i)_{1\le i, j \le \ell}=
(x_1\cdots x_\ell)^{-\ell}\prod_{i=1}^\ell (x_i^2-1)
\prod_{1\le i <j\le \ell} \left( (x_i-x_j)(1-x_ix_j)\right)
$$
but this is not needed here.

\medskip 
\noindent $\bullet$
Assume now that $m=2\ell$ is even.  The identity~\eqref{os}
still holds, with $K(x;t)$ and $M(x)$  given
by~\eqref{K-even} and~\eqref{M-even}. Based on the study of the odd
case, it would be tempting to extract the coefficient of $x_2\cdots
x_\ell^{\ell-1}$ in this identity. However, this will not give $\Gd(0,
\ldots, 0;t)$, as both $\si=\id$ and $\si=\si_1$ (the generator of
$B_\ell$ that replaces $x_1$ by $\bx_1$) contribute to this
coefficient. But we note that each term in the equation is a multiple of 
$P(x):=\prod_{i=1}^\ell (1-x_i)$. Hence we will first divide by
$P(x)$.  Let us study  the action
of $\si\in B_\ell$ on $P(x)$, with $\si$ described as
in~\eqref{Bl-Sl}. We have:
$$
\si(P(x))= \pi\left( (1-x_1^{e_1}) \cdots (1-x_\ell^{e_\ell})\right)
=
\pi\left( P(x) \prod_{i: e_i=-1} (-\bx_i) 
\right)
=
(-1)^{\sharp \{i: e_i=-1\}}P(x) \prod_{i: e_i=-1} \bx_{\pi(i)} .
$$
Hence, denoting $e=(e_1, \ldots, e_\ell)$, $x^e= (x_1^{e_1},
\ldots, x_\ell^{e_\ell})$ and $N(x)=x_2\cdots x_\ell^{\ell-1}$,
dividing~\eqref{os} by $P(x)$ gives 
\beq\label{os-revue}
\sum_{\pi \in \Sn_\ell\atop{e\in \{-1,1\}^\ell} } 
\vareps(\pi)
\ \pi\left(  N(x^e)\, \Gd(x^e;t) \prod_{ i: e_i=-1} \bx_{i}\right)
=
\frac 1 {K(x;t)}
\left( \sum_{\pi \in \Sn_\ell \atop{e\in \{-1,1\}^\ell}} 
\vareps(\pi)
\  \pi\left(N(x^e) \prod_{i: e_i=-1} \bx_{i}  \right)\right).
\eeq
Let us now extract from the left-hand side  the
coefficient of $x_2\cdots 
x_\ell^{\ell-1}$. The argument is similar to the odd case.
If $e\not = (1, \ldots,1 )$, each monomial occurring in
$ N(x^e)\, \Gd(x^e;t) \prod_{ i: e_i=-1} \bx_{i}$ contains a
negative exponent, and thus cannot contribute. Now for  $e = (1, \ldots,1 )$,
the term $\pi\left(  N(x)\, \Gd(x;t) \right)$ only contributes if
$\pi=\id$, and then its contribution is $\Gd(0,\ldots, 0;t)$, the
length \gf\ of \invs\ avoiding \decm. We obtain 
 the following counterpart of Proposition~\ref{prop:inv-m-odd}.
\begin{Proposition}\label{prop:inv-m-even}
If $m=2\ell$,  the ordinary \gf\ of \invs\ avoiding \decm\ is
the coefficient of $x_1^0x_2^1\cdots x_\ell^{\ell-1}$ in a rational function:
 $$
\Ge_m(t):=\sum_{\tau \in \In^{(m)} } t^{|\tau|}
=
\left[x_1^0x_2^1\cdots x_\ell^{\ell-1}\right]
\frac{\det(x_j^{i-1}+\bx_j^i)_{1\le i, j \le \ell}}
{1-t(x_1+\cdots +   x_\ell+ \bx_1+\cdots +\bx_\ell)}
.
$$
Equivalently, the \emm exponential, \gf\ of these involutions is
$$
\Ge_m^{(e)}(t):=\sum_{\tau \in \In^{(m)} } \frac{t^{|\tau|}}{|\tau|!}
=
 \left[x_1^0x_2^1\cdots x_\ell^{\ell-1}\right]
\det\left((x_j^{i-1}+\bx_j^i)e^{t(x_j+\bx_j)}\right)_{1\le i, j \le \ell}.
$$
\end{Proposition}
\begin{proof}
  We have just argued that $\Ge_m(t)$ is the coefficient of
  $x_2^1\cdots x_\ell^{\ell-1}$ in the right-hand side of~\eqref{os-revue}.
There remains to evaluate the numerator in the right-hand side:
\begin{eqnarray*}
  \sum_{\pi \in \Sn_\ell \atop{e\in \{-1,1\}^\ell}} 
\vareps(\pi)
\  \pi\left(N(x^e) \prod_{i: e_i=-1} \bx_{i}  \right)
&=& \sum_{\pi \in \Sn_\ell}
\vareps(\pi)\ 
\pi \left( \sum_{e\in \{-1,1\}^\ell}
\prod_{i=1}^\ell x_i^{(i-1)e_i-\chi_{ e_i=-1} }
\right)
\\
&=& \sum_{\pi \in \Sn_\ell}
\vareps(\pi)\ 
\pi \left(
\prod_{i=1}^\ell \left( x_i^{i-1 }+ \bx_i^{i }\right)\right)
\\
&=& \det(  x_j^{i-1 }+ \bx_j^{i }).
\end{eqnarray*}
This gives the expression of $\Ge_m(t)$. Taking the corresponding
exponential \gf\ gives  $\Ge_m^{(e)}(t)$.
\end{proof}

\noindent{\bf Remark.} Again, the determinant occurring in the proposition
can be evaluated in closed form~\cite[Eq.~(2.6)]{kratt-det}, but this is not needed here.

\subsection{Determinantal expression of the series}
\label{sec:det-i}
\noindent $\bullet$ Let us assume that $m$ is odd, and return to
Proposition~\ref{prop:inv-m-odd}.
Taking the exponential \gf\ rather than the ordinary one makes the
extraction of the coefficient of $x_1\cdots x_\ell^\ell$ an
elementary task, as all variables $x_j$ decouple. The series $I_i$
defined by~\eqref{Idef} arise naturally from 
$$
[x^i]e^{t(x+\bx)}= I_i.
$$
We have:
\begin{eqnarray*}
  \Ge_m^{(e)}(t)
&=&
e^t\, \sum_{\pi\in\Sn_\ell} \vareps(\pi)
\prod_{i=1}^\ell 
\left[x_i^i\right]
\left(\left(x_{i}^{\pi(i)}-\bx_{i}^{\pi(i)}\right)e^{t(x_i+\bx_i)}\right)
\\
&=&
e^t\, \sum_{\pi\in\Sn_\ell} \vareps(\pi)
\prod_{i=1}^\ell\left( I_{i-\pi(i)}- I_{i+\pi(i)}\right)
 \\
&=& e^t\, \det \left(I_{i-j}- I_{i+j}\right)_{1\le i,j\le \ell}.
\end{eqnarray*}
We have thus recovered the first part of Theorem~\ref{thm:i}.

\medskip
\noindent $\bullet$ If  $m$ is even, we start from
Proposition~\ref{prop:inv-m-even}. Again, the variables $x_j$ decouple
in the exponential \gf:
\begin{eqnarray*}
  \Ge_m^{(e)}(t)
&=&
 \sum_{\pi\in\Sn_\ell} \vareps(\pi)
\prod_{i=1}^\ell 
\left[x_i^{i-1}\right]
\left(\left(x_{i}^{\pi(i)-1}+\bx_{i}^{\pi(i)}\right)e^{t(x_i+\bx_i)}\right)
\\
&=&
 \sum_{\pi\in\Sn_\ell} \vareps(\pi)
\prod_{i=1}^\ell\left( I_{i-\pi(i)}+ I_{i+\pi(i)-1}\right)
 \\
&=& \det \left(I_{i-j}+ I_{i+j-1}\right)_{1\le i,j\le \ell}.
\end{eqnarray*}
We have thus recovered the second part of Theorem~\ref{thm:i}.

\medskip
\noindent{\bf Remark.} The determinantal expression of $\Ge_m^{(e)}$ implies that this
series is D-finite, that is, satisfies a linear differential equation
with polynomial coefficients. However, this follows as well from the constant
term expressions of Propositions~\ref{prop:inv-m-odd} and~\ref{prop:inv-m-even} using the closure properties of D-finite
series~\cite{lipshitz-df,lipshitz-diag}.

\subsection{The number of fixed points}
We now enrich our results by taking into account the number of fixed
points, thereby recovering Theorem~\ref{thm:i-f}.
Recall from Proposition~\ref{prop:inv-tree} that the label of the \inv\ obtained
by inserting $n+1$ as a fixed point in $\tau\in \In_n^{(m)}$ is 
$(a_1+1, a_2+1, \ldots, a_\ell+1)$  if  $m$ is odd, 
$(1, a_2+1, \ldots, a_\ell+1)$ otherwise. Hence, if we keep track of
the number of fixed points by a new variable $s$, the functional
equation of Proposition~\ref{prop:eq-f-i} becomes:
  \begin{multline*}
\Gc(v;t,s)= v_1+ stv_1 \Gc(v;t,s)\chi_{m\equiv {1}}
+ stv_1 \Gc(v_2,v_2,v_3, \ldots, v_\ell;t,s) \chi_{m\equiv {0}}
\\
+t^2 v_1 \sum_{j=1}^{\ell} v_{j}v_{j+1}\ 
\frac{\Gc(v;t,s)-\Gc(v_1, \ldots, v_{j-1}, v_{j+1}, v_{j+1}, v_{j+2}, \ldots,
  v_\ell;t,s)}{v_{j}-v_{j+1}}.
\end{multline*}
The series $\Gc(1, \ldots, 1;t,s)$ counts \invs\ of   $\In^{(m)}$ by
their length and number of fixed points. The change of
variables~\eqref{vx} now gives
\begin{multline*}
  K(x;t,s) \Gd(x;t,s)=1+ st \Gd(0,x_2, \ldots, x_\ell;t,s)\chi_{m\equiv 0}
\\ - t \sum_{j=1}^\ell
\bx_j \Gd(x_1, \ldots, x_{j-2},x_{j-1}+x_j, 0, x_{j+1}, \ldots,
x_\ell;t,s),
\end{multline*}
where
$$
K(x;t,s)=1-t(x_1+\cdots+ x_\ell)-st\chi_{m\equiv 1}-t(\bx_1+\cdots+
\bx_\ell),
$$
and
$$
\Gd(x;t,s)\equiv \Gd(x_1, \ldots, x_\ell;t)=\Gc\left(\frac 1{1-t(x_1+\cdots+
  x_\ell)},  \frac 1{1-t(x_2+\cdots+ x_\ell)},\ldots, 
\frac 1{1-tx_\ell};t,s\right).
$$

\noindent $\bullet$ If $m$ is odd, the argument of Sections~\ref{sec:orbit-i},~\ref{sec:extr-i},~\ref{sec:det-i}, applies
verbatim. The only difference is that the term $t$ occurring in the
kernel is replaced by $st$. This gives at once the first part of
Theorem~\ref{thm:i-f}, in the form~\eqref{inv-fp}. 
\medskip

\noindent $\bullet$ If $m$ is even, the equation reads:
\begin{multline*}
  K(x;t) \Gd(x;t,s)=1+ t(s-\bx_1) \Gd(0,x_2, \ldots, x_\ell;t,s)
 \\- t \sum_{j=2}^\ell
\bx_j \Gd(x_1, \ldots, x_{j-2},x_{j-1}+x_j, 0, x_{j+1}, \ldots,
x_\ell;t,s)
\end{multline*}
with $K(x;t)=1-t(x_1+\cdots+ x_\ell+\bx_1+\cdots+\bx_\ell)$. We
multiply it by 
$$
M(x;s):= x_2x_3^2\cdots x_\ell^{\ell-1}(s-x_1)\cdots(s-x_\ell),
$$
and then argue as in Section~\ref{sec:orbit-i} to conclude
that
\beq\label{os-f}
\sum_{\si \in B_\ell} \vareps(\si)\  \si\! \left(  M(x;s)
\Gd(x;t,s)\right)
=\frac{\sum_{\si \in B_\ell} \vareps(\si)\ \si (M(x;s))}{K(x;t)},
\eeq
 with the above  values of $K(x;t)$ and $M(x;s)$.

Now we cannot follow exactly the argument of
Section~\ref{sec:extr-i}, because $\si(M(x;s))$ does not differ from
$M(x;s)$ by a monomial. So it does not help to divide the
equation by $(s-x_1)\cdots(s-x_\ell)$. Instead, let us leave the equation as
it is, and extract all terms of the form $x_1^ax_2^1\cdots
x_\ell^{\ell-1}$ with $a\ge 0$. More precisely,  for a series $F(x_1,\ldots, x_\ell;t,s)$ in
$\qs[x_1,\ldots, x_\ell,s][[t]]$, let us denote
\beq\label{not}
[x_1^{\ge 0} x_2^{1}\cdots x_\ell^{\ell-1}] F(x_1,\ldots, x_\ell;t,s):=
\sum _{a \ge 0} x_1^{a}\, [x_1^{a} x_2^{1}\cdots
  x_\ell^{\ell-1}]F(x_1,\ldots, x_\ell;t,s).
\eeq
Consider the term 
$$
\si\!\left(  M(x;s)\Gd(x;t,s)\right)=
\si\!\left( x_2x_3^2\cdots x_\ell^{\ell-1}(s-x_1)\cdots(s-x_\ell)\Gd(x;t,s)\right).
$$ 
Let us decouple in
$\si$ the sign changes $e_1, \ldots, e_\ell$ and the \p\ $\pi$ of the
$x_i$'s, as 
in~\eqref{Bl-Sl}. We wish to determine $[x_1^{\ge 0} x_2^{1}\cdots
  x_\ell^{\ell-1}]\,\si\!\left(  M(x;s)\Gd(x;t,s)\right)$.
\begin{itemize}
\item [--]
If one of the $e_i$'s, for $i\ge 2$, is $-1$, then
all monomials occurring in  $\si\!\left(  M(x;s)\Gd(x;t,s)\right)$ involve a negative exponent and
thus do not contribute. 
\item[--] If $e_1=-1$ while $e_i=1$ for $i \ge 2$, the
only way to obtain a non-zero contribution of $\si\left(
M(x;s)\Gd(x;t,s)\right)$ is to take $\pi=\id$, and the contribution is
then 
$$
s^\ell \Gd(0,\ldots, 0;t,s).
$$
 \item[--]
If $\si=\pi\in \Sn_\ell$,  the contribution is
$$
 (s-x_1) [x_1^{\ge 0} x_2^{1}\cdots x_\ell^{\ell-1}]  
\left((s-x_2)\cdots (s-x_\ell)\ \pi(x_2^{1}\cdots x_\ell^{\ell-1}\Gd(x;t,s))\right).
$$
We note that this is a multiple of $(s-x_1)$. 
\end{itemize}
Hence, the result of our coefficient extraction on~\eqref{os-f} is
\begin{multline*}
  -s^\ell \Gd(0,\ldots, 0;t,s)
+ (s-x_1)\sum_{\pi\in \Sn_\ell} \vareps(\pi) [x_1^{\ge 0} x_2^{1}\cdots x_\ell^{\ell-1}]  
\left( (s-x_2)\cdots (s-x_\ell)\ \pi(x_2^{1}\cdots x_\ell^{\ell-1}\Gd(x;t,s))\right)
\\
= [x_1^{\ge 0} x_2^{1}\cdots x_\ell^{\ell-1}]\,
\frac{\sum_{\si \in B_\ell} \vareps(\si)\ \si
  (M(x;s))}{K(x;t)}.
\end{multline*}
Let us specialize this to $x_1=s$:
$$
 -s^\ell \Gd(0,\ldots, 0;t,s)=\left( [x_1^{\ge 0} x_2^{1}\cdots x_\ell^{\ell-1}]
\left.
\frac{\sum_{\si \in B_\ell} \vareps(\si)\ \si
  (M(x;s))}{K(x;t)}\right)\right|_{x_1\mapsto s}.
$$
The kernel $K(x;t)$ is independent of $s$. But this is also the case of
\begin{eqnarray*}
  \sum_{\si \in B_\ell} \vareps(\si)\ \si   (M(x;s))
&=&
\sum_{\pi \in \Sn_ \ell}\vareps(\pi) \pi \left(
\prod_{i=1}^\ell ( (s-x_i)x_i ^{i-1} -  (s-\bx_i)\bx_i ^{i-1})
\right)
\\
&=&
\det\left(s(x_j ^{i-1}-\bx_j ^{i-1}) - x_j ^{i}+ \bx_j
^{i}\right)_{1\le i, j \le \ell}
\\
&=&
\det\left( - x_j ^{i}+ \bx_j^{i}\right)_{1\le i, j \le \ell}
\end{eqnarray*}
as is seen by taking linear combinations of rows. We have thus obtained the
following counterpart of Proposition~\ref{prop:inv-m-even}.
\begin{Proposition}\label{prop:inv-m-even-f}
If $m=2\ell$,  the ordinary \gf\ of \invs\ avoiding \decm,
counted by the length and number of fixed points, is, with the
notation~\eqref{not}: 
 $$
\Ge_m(t,x_1):=\sum_{\tau \in \In^{(m)} } t^{|\tau|} x_1^{f(\tau)}
=
-\frac{1}{x_1^\ell} 
[x_1^{\ge 0} x_2^1\cdots x_\ell^{\ell-1}]
\frac{\det( \bx_j^{i}- x_j ^{i})_{1\le i, j \le \ell}}
{1-t(x_1+\cdots +   x_\ell+ \bx_1+\cdots
  +\bx_\ell)}.
$$
Equivalently, the \emm exponential, \gf\ of these involutions is
$$
\Ge_m^{(e)}(t,x_1):=\sum_{\tau \in \In^{(m)} } \frac{t^{|\tau|}}{|\tau|!}
\ x_1^{f(\tau)}=-\frac{1}{x_1^\ell}
 [x_1^{\ge 0}x_2^1\cdots x_\ell^{\ell-1}]
\det\left(( \bx_j^{i}- x_j ^{i})e^{t(x_j+\bx_j)}\right)_{1\le i, j \le
  \ell}
.
$$
\end{Proposition}
We can now perform the coefficient extraction explicitly in the
expression of $\Ge_m^{(e)}(t,x_1)$:
\begin{eqnarray*}
  \Ge_m^{(e)}(t,x_1)
&=&
-\frac{1}{x_1^\ell}\sum_{\pi\in \Sn_\ell} \vareps(\pi) 
[x_1^{\ge 0}]\left( (\bx_1^{\pi(1)}- x_1^{\pi(1)})e^{t(x_1+\bx_1)}\right)
\prod_{i=2}^\ell [x_i^{i-1}]\left( ( \bx_i^{\pi(i)}- x_i ^{\pi(i)})e^{t(x_i+\bx_i)}
\right)\\
&=&
-\frac{1}{x_1^\ell}\sum_{\pi\in \Sn_\ell} \vareps(\pi) 
\sum_{k\ge 0} x_1^k (I_{k+\pi(1)} - I_{k-\pi(1)})\prod_{i=2}^\ell (I_{i+\pi(i)-1} - I_{i-\pi(i)-1})
\\
&=&
\sum_{k\ge 0} x_1^{k-\ell}   \det \left(
\begin{array}{c}
 (I_{k-j}-I_{k+j} )_{1\le j \le \ell}
\\
 (I_{i+j-1} - I_{i-j-1})_{2\le i \le \ell, 1\le j \le \ell}
\end{array}\right).
\end{eqnarray*}
Upon extracting the coefficient of $x_1^p$, this gives the second part
of Theorem~\ref{thm:i-f}.

\section{Permutations  with no long ascending subsequence}
\label{sec:perm}

We now want to derive from the functional equation of
Proposition~\ref{prop:eq-p} the Bessel \gf\ of \ps\ avoiding \incm, 
given in Theorem~\ref{thm:p}.
We follow the same steps as in the case of involutions, but the
coefficient extraction is  more delicate.

\subsection{Invariance properties of the kernel}
\label{sec:inv-p}
The kernel of the equation of Proposition~\ref{prop:eq-p}, that is,
the coefficient of $\Fc(v;t)$, reads
$$
1-tv_1- t\sum_{j=2}^m \frac{v_{j-1}v_j}{v_{j-1}-v_j}.
$$
Its invariance properties appear clearly if we set
$$
v_j=\frac 1{x_1+\cdots + x_j}.
$$
Indeed, the kernel then becomes
\beq\label{K-p}
K(x;t):= 1-t (\bx_1+\cdots + \bx_m),
\eeq
with $\bx_i=1/x_i$, 
and is  invariant under the action of the symmetric group $\Sn_m$, seen as a group
of transformations of Laurent polynomials in the $x_i$. This group is
generated by $m-1$ elements of order~2, denoted $\si_1, \ldots,
\si_{m-1}$:
$$
\si_j(P(x_1,  \ldots, x_m))= P(x_1,\ldots, x_{j-1},  x_{j+1},x_j,
x_{j+2}, \ldots, x_m).
$$
The functional equation now reads
$$
K(x;t)\Fd(x;t) = 1- t  \sum_{j=1}^{m-1} \bx_{j+1} \Fd(x_1, \ldots,  x_{j-1}, x_{j}+x_{j+1}, 0,
x_{j+2}, \ldots, x_m;t),
$$
with 
\beq\label{fd-def}
\Fd(x;t)\equiv \Fd(x_1, \ldots, x_m;t)=
\Fc\left(\frac 1 {x_1}, \frac 1 {x_1+x_2}, \ldots,  \frac 1
        {x_1+\cdots + x_m};t\right).
\eeq

\subsection{Orbit sum}
\label{sec:orbit-p}
Let us multiply the equation by 
\beq\label{M-p}
M(x)=x_1^0 x_2^1\cdots x_m^{m-1}.
\eeq
 This gives:
\begin{multline}\label{Fd-sym}
K(x;t) M(x)  \Fd(x;t)= M(x)\\
-  t \sum_{j=1}^{m-1}
x_1^0 \cdots  x_j^{j-1} x_{j+1}^{j-1} x_{j+2}^{j+1} \cdots x_m^{m-1} 
\Fd(x_1, \ldots, x_{j-1},x_{j}+x_{j+1}, 0, x_{j+2}, \ldots,x_m;t).
\end{multline}
 The $j^{\hbox{\small th}}$ term of the sum is invariant under the
 action of the
generator $\si_j$ (which exchanges $x_{j}$ and $x_{j+1}$).
Consequently, forming the signed sum of~\eqref{Fd-sym} over the
symmetric group $\Sn_m$ gives the following orbit sum:
$$
\sum_{\si \in\Sn_m } \vareps(\si)\  \si \left( K(x;t) M(x)
\Fd(x;t)\right)
=\sum_{\si \in \Sn_m} \vareps(\si) \si ( M(x)),
$$
or, given that $K(x;t)$ is $\Sn_m$-invariant,
\beq\label{os-p}
\sum_{\si \in \Sn_m} \vareps(\si)\  \si \left(  M(x)
\Fd(x;t)\right)
=\frac{\sum_{\si \in \Sn_m} \vareps(\si)\ \si (M(x))}{K(x;t)}
=\frac{\det(x_j^{i-1})_{1\le i, j\le m}}{K(x;t)},
\eeq
where $K(x;t)$ is given by~\eqref{K-p} and $M(x)$
by~\eqref{M-p}.

\subsection{Extraction of $\Fc(1, \ldots, 1;t)$}
\label{sec:extr-p}

For $1\le j \le m$, let us now denote $z_j=x_1+\cdots
+x_j$. Equivalently, $x_j=z_j-z_{j-1}$ with $z_0=0$. All
series occurring in the orbit sum~\eqref{os-p} become  series 
in $t$ with coefficients in $\qs(z_1, \ldots, z_m)$. In particular,
$$
\Fd(x;t)=
\Fc\left(\frac 1 {z_1}, \frac 1 {z_2}, \ldots,  \frac 1
        {z_m};t\right)
$$
has coefficients which are \emm Laurent polynomials, in the $z_j$'s. This is
not the case for all terms in~\eqref{os-p}. For instance, if
$\si$ is the 2-cycle $(1,2)$,
$$
\si (\Fd(x;t))=\Fc\left(\frac 1 {x_2}, \frac 1 {x_1+x_2}, \ldots,  \frac 1
        {x_1+\cdots + x_m};t\right)=
\Fc\left(\frac 1 {z_2-z_1}, \frac 1 {z_2}, \ldots,  \frac 1
        {z_m};t\right)
$$
involves coefficients which are \emm not, Laurent polynomials in the $z_j$'s.
In order to
perform our  extraction, we will expand all rational functions of the
$z_j$'s as  (iterated)
Laurent series, by expanding first in $z_1$, then in $z_2$, and so
on. For instance, the expansion of $1/(x_1+x_3+x_4)$ reads
$$
\frac 1 {x_1+x_3+x_4} =\frac 1 {z_4-z_2+z_1}
= \sum_{e_1 \ge 0} \frac {(-z_1)^{e_1}}{(z_4-z_2)^{e_1+1}}
= \sum_{e_1 \ge 0,e_2\ge 0} 
{e_1+e_2 \choose e_1} \frac{(-z_1)^{e_1}z_2^{e_2}}{z_4^{1+e_1+e_2}}.
$$
In other words, the coefficients of our series in $t$ now lie in the ring of
\emm iterated
Laurent series, in $z_1, \ldots, z_m$, which 
is defined inductively as follows:
\begin{itemize}
\item [--] if $m=1$, it coincides with the ring of Laurent
  series\footnote{Recall that our Laurent series only involve finitely
    many negative exponents.} in
  $z_1$ (with rational coefficients),
\item [--] if $m>1$, it is  the ring of Laurent series in
  $z_1$ whose coefficients are iterated Laurent series in $z_2,
  \ldots, z_m$.
\end{itemize}
It follows from this definition that an iterated Laurent series in the
$z_j$'s only contains finitely many \emm non-positive, monomials, that
is, monomials $z_1^{e_1} \cdots z_m^{e_m}$ with  $e_j \le 0$ for all
$j$. This allows us to  define below a linear operator $\Lambda$, which extracts 
from an iterated Laurent series \emm some, coefficients  associated
with non-positive monomials  and adds them up.
\begin{Definition}
\label{def:Lambda} 
Let $\Lambda$ be the linear operator defined on  iterated
  Laurent series in $z_1, \ldots, z_m$ by the following action  on monomials:
  \beq\label{Lambda-def}
\Lambda( z_1^{e_1} \cdots z_m^{e_m} )
=\left\{
\begin{array}{ll}
  1, & \hbox{if } e_1 \le 0, \ldots,  e_m \le 0 \hbox{ and } e_j =0
  \Rightarrow e_{j+1}=\cdots = e_m=0;
\\
0, &\hbox{otherwise}.
\end{array}
\right.
\eeq
\end{Definition}
\noindent{\bf Remark.}
The action of $\Lambda$ can also be described as the extraction of a constant
term: for any iterated Laurent series $F(z_1, \ldots, z_m)$,
$$
\Lambda (F(z_1, \ldots, z_m))=
[z_1^0 \cdots z_m^0]\left( F(z_1, \ldots, z_m) \sum_{i=0}^m 
\prod_{j=1}^i \frac{z_j}{1-z_j}\right).
$$

This operator has been designed to extract from~\eqref{os-p}
the series $\Fc(1,  \ldots, 1; t)$ in which we are interested. The
following proposition is thus the counterpart of Propositions~\ref{prop:inv-m-odd}
and~\ref{prop:inv-m-even}.
\begin{Proposition}\label{prop:p-extr}
  The ordinary \gf\ of \ps\ avoiding \incm\ is
obtained by applying $\Lambda$ to  a rational function:
 $$
\Fe_m(t):=\sum_{\tau \in \Sn^{(m)} } t^{|\tau|}
=
\Lambda\left(
\frac{\det(x_j^{i-j})_{1\le i, j \le m}}
{1-t(\bx_1+\cdots +   \bx_m)}
\right),
$$
with $x_j=z_j-z_{j-1}$ and $z_0=0$.

Equivalently, the \emm exponential, \gf\ of these \ps\ is
$$
\Fe_m^{(e)}(t):=\sum_{\tau \in \Sn^{(m)} } \frac{t^{|\tau|}}{|\tau|!}
=
\Lambda\left(
{\det(x_j^{i-j}e^{t\bx_j})_{1\le i, j \le m}}
\right).
$$
\end{Proposition}
\noindent{\bf Remarks}\\
1. The fact that the action of $\Lambda$ can be described as a
constant term extraction, combined with closure properties of D-finite
series~\cite{lipshitz-df,lipshitz-diag}, implies that the series $\Fe_m$ (and
$\Fe_m^{(e)}$) are D-finite. This was first proved by
Gessel~\cite{gessel-symmetric}. 
\\
2. Again, the determinant is a Vandermonde determinant and can be
evaluated in closed form, but this will not be needed.
\begin{proof}
We will prove that for all $\si \in \Sn_m$,
\beq\label{Lambda-F}
\Lambda \left(\frac { \si \left(  M(x)\Fd(x;t)\right)} {M(x)} \right)=
\left\{
\begin{array}{ll}
  \Fc(1,  \ldots, 1; t), & \hbox{if } \si=\id;\\
0, & \hbox{otherwise,}
\end{array}
\right.
\eeq
so that the first result directly follows from~\eqref{os-p}, after
dividing by $M(x)$ and applying $\Lambda$.  
It is then easily converted into an expression for the exponential \gf.

Recall the definition~\eqref{fd-def} of $\Fd(x;t)$,
and that $\Sn_m$ acts by permuting the $x_j$'s.
Also, 
$$
\Fc(v_1, \ldots, v_m;t)= \sum_{\tau \in \Sn^{(m)}}
v_1^{a_2(\tau)-1} v_2^{a_3(\tau)-a_2(\tau)} \cdots v_m^{|\tau|+1-a_m(\tau)} t^{|\tau|},
$$
where  the labels $a_i(\tau)$ are defined by~\eqref{aj-def}.  This definition implies that, if $a_{j}(\tau)=
a_{j+1}(\tau)$, then   $a_j(\tau)=a_{j+1}(\tau)=\cdots = a_{m}(\tau)=|\tau|+1$.
In other words, the $v$-monomials occurring in $\Fc(v;t)$ satisfy a
property that should be reminiscent of the definition of $\Lambda$:
\beq\label{Lambda-prop}
\Fc(v_1, \ldots, v_m;t)= \sum_{ (e_1, \ldots, e_m)\in \cE} c(e_1,
\ldots, e_m)\,
v_1^{e_1} v_2^{e_2} \cdots v_m^{e_m} t^{e_1+\cdots + e_m},
\eeq
where 
$$
\cE=\{(e_1, \ldots, e_m)\in \ns^m : e_j=0 \Rightarrow e_{j+1}=\cdots
= e_m=0\}.
$$

\medskip
With this property at hand, we can now address the proof of~\eqref{Lambda-F}.
If $\si=\id$, 
$$
\Lambda \left(\frac { \si \left(  M(x)\Fd(x;t)\right)} {M(x)} \right)=
\Lambda \left( \Fc\Big(\frac  1 {z_1}, \frac 1 {z_2}, \ldots,  \frac 1
        {z_m};t\Big)\right) = \Fc(1,  \ldots, 1; t)
$$
by definition of $\Lambda$ and~\eqref{Lambda-prop}. 

There remains to prove  the second part of ~\eqref{Lambda-F}. Let us
consider an example, say $m=5$ and $\si=13425$. 
Let $\tau\in \Sn^{(5)}_n$, and denote  $e_i=
a_{i+1}(\tau)-a_i(\tau)$, with $a_1(\tau)=1$ and
$a_{m+1}(\tau)=|\tau|+1$. Of course, $e_i\ge 0$ for all $i$. Up to a
factor $t^{|\tau|}$, the contribution of
$\tau$ in  $\si \left(  M(x)\Fd(x;t)\right)/ M(x)$ is 
\begin{multline*}
  \frac 1 {x_2x_3^2x_4^3x_5^4}\, \si \left( 
\frac{x_2x_3^2x_4^3x_5^4}{x_1^{e_1} (x_1+x_2)^{e_2} \cdots (x_1+\cdots + x_5)^{e_5}}
\right)
=\\
\frac{x_3x_4^2x_2^3x_5^4}{x_2x_3^2x_4^3x_5^4\,x_1^{e_1} (x_1+x_3)^{e_2} (x_1+x_3+x_4)^{e_3}  (x_1+x_2+x_3+x_4)^{e_4}  (x_1+\cdots + x_5)^{e_5}}
=\\
\frac{(z_2-z_1)^2}{(z_3-z_2)(z_4-z_3)z_1^{e_1} (z_3-z_2+z_1)^{e_2} (z_4-z_2+z_1)^{e_3}
  z_4^{e_4}  z_5^{e_5}}.
\end{multline*}
To prepare the Laurent expansion in the variables $z_i$, we
rewrite this fraction as
\begin{multline*}
\frac{(z_2-z_1)^2}
{
z_1^{e_1} z_3^{1+e_2}z_4^{1+e_3+e_4}  z_5^{e_5}
\left(1-\frac{z_3}{z_4}\right)
\left(1-\frac{z_2}{z_3}\right)^{1+e_2}
\left(1-\frac{z_2}{z_4}\right)^{e_3}
\left(1+\frac{z_1}{z_3\left(1-\frac{z_2}{z_3}\right)}\right)^{e_2} 
\left(1+\frac{z_1}{z_4\left(1-\frac{z_2}{z_4}\right)}\right)^{e_3}
 }.
\end{multline*}
It is now clear that, in each term of the Laurent expansion,
 $z_2$ has a non-negative exponent, while  $z_4$ has a negative exponent.
By definition of $\Lambda$, this implies that 
$$
\Lambda\left(\frac 1 {M(x)} \ \si \left( \frac{ M(x)}{x_1^{e_1} (x_1+x_2)^{e_2} \cdots (x_1+\cdots + x_5)^{e_5}}\right)\right)=0.
$$
As this holds for every \p\ $\tau\in \Sn^{(5)}_n$,
we have proved that~\eqref{Lambda-F} holds for $\si=13425$.

Let us say that a series of $\qs(x_1, \ldots, x_m)[[t]]$ is
 positive in $z_j$ (or $z_j$-positive) if,
in every term of its $z$-expansion, $z_j$ appears with a positive
exponent. We define similarly the notion of $z_j$-negative,
$z_j$-non-positive, $z_j$-non-negative series. 
We have just observed that, for $m=5$ and $\si=13425$, the series  $\si \left(
M(x)\Fd(x;t)\right)/ M(x)$ is non-negative in $z_2$ but negative in
$z_4$. This is generalized by the following lemma.
\begin{Lemma}\label{lem:extr}
  Take $\si\in \Sn_m\setminus\{\id\}$.  Let $\si(j)$ be the largest non-fixed
  left-to-right maximum of $\si$. That is,
$$
\hbox{for } k<j,  \si(k)<\si(j), \hbox{ and for every k such that }
 \si(k)>\si(j), \hbox{ one has }\si(k)=k.
$$
Let $\si(i)$ be any value that is not a left-to-right maximum and
satisfies $\si(i)\le i$.
For $e_1\ge 0, \ldots, e_m \ge 0$, consider the fraction
\beq\label{frac}
\frac 1 {M(x)}\ \si \!\left( 
\frac{M(x)}{x_1^{e_1} (x_1+x_2)^{e_2} \cdots (x_1+\cdots + x_m)^{e_m}}
\right).
\eeq
Then this fraction is non-negative in $z_{\si(i)}$ but negative in
$z_{\si(j)}$. Since $\si(i)<\si(j)$, applying $\Lambda$
to this fraction gives $0$.
\end{Lemma}
Returning to the example $\si=13425 $ studied above, we observe that
the lemma applies with $\si(i)=2$ and $\si(j)=4$. 

This lemma implies the second part of~\eqref{Lambda-F}: indeed, the
contribution of any $\tau\in \Sn^{(m)}$ in $\si(M(x)\Fd(x;t))/M(x)$ is of the
form~\eqref{frac}. Hence proving the lemma will  conclude the proof of
Proposition~\ref{prop:p-extr}. 
\end{proof}

\medskip
\noindent{\em Proof of Lemma~}\ref{lem:extr}.
  We establish this lemma via a sequence of three elementary properties.
  \begin{Property}\label{P1}
           Let $i_1 <i_2<\cdots <i_k$, and $e \in \zs$. The fraction
$$
\frac 1{(\pm z_{i_1}\pm \cdots \pm z_{i_k})^e}
$$
is non-negative in $z_{i_1}, \ldots , z_{i_{k-1}}$. If $e\ge 0$, it is
non-positive in $z_{i_k}$, and even negative in  $z_{i_k}$ if $e>0$.
  \end{Property}
  \begin{proof}
 The result is obvious if $e\le 0$, as the fraction is a polynomial in
this case. If $e>0$, we prove it by induction on $k$. It clearly holds
for $k=1$. If $k>1$, we write
\begin{multline*}
  \frac 1{(\pm z_{i_1}\pm \cdots \pm z_{i_k})^e}= \frac 1
{(\pm z_{i_2}\pm \cdots \pm z_{i_k})^e\left( 1 \pm \frac{z_{i_1}}{z_{i_2}\pm
  \cdots \pm z_{i_k}}\right)^e}\\
=\sum_{n\ge 0} {e-1+n \choose n} \frac{(\pm z_{i_1})^n}{(\pm z_{i_2}\pm
  \cdots \pm z_{i_k})^{e+n}},
\end{multline*}
and conclude by induction on $k$.
    \end{proof}
 \begin{Property}\label{P2}
    Let  $\si $, $j$,  $e_1,    \ldots, e_m$ be as in
    Lemma~{\rm\ref{lem:extr}}. The fraction
$$
 \si \left( 
\frac{1}{x_1^{e_1} (x_1+x_2)^{e_2} \cdots (x_1+\cdots + x_m)^{e_m}}
\right)
$$
is non-negative  in all $z_{\si(k)}$ such that $\si(k)$ is not a
left-to-right maximum, and non-positive in $z_{\si(j)}$.
  \end{Property}
 \begin{proof}
 It suffices to prove that the result holds for each term
\beq\label{term}
 \si \left( 
\frac{1}{ (x_1+\cdots + x_\ell)^{e_\ell}}
\right)= 
\frac 1 {(z_{\si(1)}- z_{\si(1)-1} + \cdots + z_{\si(\ell)}- z_{\si(\ell)-1})^{e_\ell}},
\eeq
for  $\ell \in \{1,\ldots,  m\}$ (with $z_0=0$). 

   By Property~\ref{P1}, this term 
is non-negative in all variables, except possibly in $z_{\max(\si(1),
  \ldots, \si({\ell}))}$. Since $\max(\si(1),
  \ldots, \si({\ell}))$ is always a
left-to-right maximum, this proves the first part of the property.

Consider now the variable $z_{\si(j)}$.
\begin{itemize}
\item [--]
If $\ell<j$, then $\max(\si(1),
  \ldots, \si({\ell}))< \si(j)$, so that the term~\eqref{term}  is independent of $z_{\si(j)}$, and thus
  non-positive in this variable.
\item  [--] If $j\le \ell \le \si(j)$, then $\max(\si(1),
  \ldots, \si({\ell}))= \si(j)$. 
Then~\eqref{term} is non-positive in
  $z_{\si(j)}$  by Property~\ref{P1}. 
\item[--] 
 Finally, if $\ell>\si(j)$, then  $\{\si(1),
  \ldots, \si({\ell})\}=\{1, \ldots, \ell\}$, so that the
  term~\eqref{term} simply reads $1/z_{\ell}^{e_\ell}$. This  is
  independent of  $z_{\si(j)}$, and thus  non-positive in this variable.
\end{itemize}
    \end{proof}%
 \begin{Property}\label{P3}
      Let  $\si $ and $j$ be as in   Lemma~{\rm\ref{lem:extr}}. The fraction
$$
\frac {\si \left( M(x)\right)}{M(x)}
$$
is non-negative  in all $z_{\si(k)}$ such that $\si(k)\le k$, and
negative in $z_{\si(j)}$.  
  \end{Property}
 \begin{proof}
    We have 
$$
\frac {\si \left( M(x)\right)}{M(x)}= \prod_{\ell=1}^m
(z_{\si(\ell)}-z_{\si(\ell)-1})^{\ell-\si(\ell)} .
$$
Assume $\si(k)\le k$. The two terms of the above product that involve
$z_{\si(k)}$ are 
$(z_{\si(k)}-z_{\si(k)-1})^{k-\si(k)}$ and
$(z_{\si(k)+1}-z_{\si(k)})^{e}$, with $e= \si^{-1}(\si(k)+1)-\si(k)-1$.
The former term is non-negative in $z_{\si(k)}$ because  $\si(k)\le k$. The
latter term is non-negative in $z_{\si(k)}$  by
Property~\ref{P1}. This proves the first part of the property.

The two terms  that involve
$z_{\si(j)}$ are  $(z_{\si(j)}-z_{\si(j)-1})^{j-\si(j)}$ and 
$(z_{\si(j)+1}-z_{\si(j)})^{e}$, with $e=
\si^{-1}(\si(j)+1)-\si(j)-1$. Since $\si(j)>j$, the former term is
negative in $z_{\si(j)}$ by Property~\ref{P1}. By construction of $j$,
the exponent $e$ is $0$, so that the latter term is simply $1$.
    \end{proof}

\medskip
Lemma~\ref{lem:extr} now follows by combining  Properties~\ref{P2}
and~\ref{P3}.
\qed

\subsection{Determinantal expression of the series}
\label{sec:det-p}
Let us write $e^{t\bx}= \sum_{b\ge 0} (t\bx)^b/b!$. The second formula
in Proposition~\ref{prop:p-extr} reads: 
\beq\label{F-expr}
\sum_{\tau \in \Sn^{(m)}} \frac {t^{|\tau|}}{|\tau|!}=
\sum_{b_1, \ldots, b_m \ge 0}\frac{t^{b_1+\cdots + b_m}}{b_1! \cdots b_m!}
\sum_{\si \in \Sn_m} \vareps(\si) \ \Lambda \!
\left( \frac {\si (M(x))}{M(x) \si(x^b)}\right)
\eeq
where $M(x)= x_2x_3^2\cdots x_m ^{m-1}$, $b=(b_1, \ldots, b_m)$, and
$x^b=x_1^{b_1} \cdots x_m^{b_m}$.  
We will give a closed form expression of $\Lambda \!
\left( \frac {\si (M(x))}{M(x) \si(x^b)}\right)$ (Lemma~\ref{extr-si}), which
in turn will give a  closed form expression of the sum
over $\si$ occurring in~\eqref{F-expr}.
\begin{Proposition}\label{prop:eval-sum}
For $b=(b_1, \ldots, b_m) \in \ns^m$,
$$
\sum_{\si \in \Sn_m} \vareps(\si) \ \Lambda \!
\left( \frac {\si (M(x))}{M(x) \si(x^b)}\right)
=
\frac{(b_1+\cdots + b_m)!}{\prod_{i=1}^m (b_i-i+m)!}
\prod_{1\le i <j \le m} (b_i-i-b_j+j).
 $$
\end{Proposition}

Let us delay for the moment the proof of this proposition, and derive
from it  Gessel's determinantal formula (Theorem~\ref{thm:p}).

\noindent{\em Proof of Theorem~}~\ref{thm:p}.
The exponential \gf\ of \ps\ of $\Sn^{(m)}$ now reads
$$
\sum_{\tau \in \Sn^{(m)}} \frac {t^{|\tau|}}{|\tau|!}=
\sum_{b_1, \ldots, b_m \ge 0}{t^{b_1+\cdots + b_m}}
\frac{(b_1+\cdots + b_m)!}{\prod_{i=1}^m b_i!(b_i-i+m)!}
\prod_{1\le i <j \le m} (b_i-i-b_j+j).
$$
Replacing $t$ by $t^2$, and taking the Bessel \gf\ gives
\beq\label{perm-final}
\sum_{\tau \in \Sn^{(m)} } \frac {t^{2|\tau|}}{|\tau|!^2}=
\sum_{b_1, \ldots, b_m \ge 0}
\frac{t^{2(b_1+\cdots + b_m)}}{\prod_{i=1}^m b_i!(b_i-i+m)!}
\prod_{1\le i <j \le m} (b_i-i-b_j+j).
\eeq
But this is exactly Gessel's determinant. Indeed:
\begin{eqnarray*}
  \det (I_{j-i})_{1\le i, j\le m}
&=& \sum_{\si \in \Sn_m} \vareps(\si)
\prod_{i=1}^m I_{\si(i)-i}\\
&=& \sum_{\si \in \Sn_m} \vareps(\si)
\prod_{i=1}^m \sum_{b_i\ge 0} \frac{t^{2b_i-i+\si(i)}}{b_i!
  (b_i-i+\si(i))!}\\
&=&
\sum_{b_1, \ldots, b_m \ge 0}\frac{t^{2(b_1+\cdots + b_m)}}
{\prod_{i=1}^m b_i!(b_i-i+m)!}\det
\left((b_i-i+j+1)\cdots(b_i-i+m)\right),
\end{eqnarray*}
and this coincides with~\eqref{perm-final}, because the above determinant is the
Vandermonde determinant in the variables $u_i:=b_i-i$ (since $(b_i-i+j+1)\cdots(b_i-i+m)$ is a polynomial in
$u_i$ of dominant term $u_i^{m-j}$).
\qed

\medskip
There remains to prove Proposition~\ref{prop:eval-sum}. The proof relies on two lemmas. The first one is a simple identity based on a
partial fraction expansion. The second gives a closed form expression
of  $\Lambda\left( \frac {\si  (M(x))}{M(x) \si(x^b)}\right)$, for $b \in \ns^m$.

\begin{Lemma}\label{lem:identite}
  Let $x_1, \ldots, x_k, u_1, \ldots, u_k$ be indeterminates, and let
  the symmetric group $\Sn_k$ act on the $x_i$'s by permuting them
  (that is, $\tau(x_i)= x_{\tau(i)}$ for $\tau \in \Sn_k$).
Then
$$
\sum_{\tau\in \Sn_k} \vareps(\tau) \ 
\tau  \left( 
\prod_{i=1}^{k-1}
\frac{(x_i+u_i) \cdots (x_i+u_k)}
{x_i+u_i+ \cdots + x_k+u_k}
\right)
=\prod_{1\le i<j\le k} (x_i-x_j).
$$
\end{Lemma}
\begin{proof}
  Let us denote $u=(u_1, \ldots, u_k)$ and
$$
T(x,u)=\prod_{i=1}^{k-1}
\frac{(x_i+u_i) \cdots (x_i+u_k)}
{x_i+u_i+ \cdots + x_k+u_k}. 
$$
This is a rational function of $u_k$, in which  the
  numerator and denominator have degree $k-1$.  By a partial
  fraction expansion,
\beq\label{T-C}
T(x,u)= C(x,u)+ \sum_{\ell=1}^{k-1} \frac{\al_\ell(x,u)}{x_\ell+u_\ell+ \cdots + x_k+u_k},
\eeq
where $C$ and the $\al_\ell$'s are independent of $u_k$.
By letting $u_k$ tend to infinity, one obtains
$$
C(x,u)= \prod_{1\le i\le j <k} (x_i+u_j).
$$
The value of $\al_\ell$ is obtained by taking the residue of $T(x,u)$
at $u_k=-(x_\ell+u_\ell+ \cdots + x_k)$. This gives, for $\ell\le k-1$:
$$
\al_\ell(x,u)= \frac{\prod_{1\le i\le j <k} (x_i+u_j)
\prod_{1\le i < k}(x_i - (x_\ell+u_\ell+ \cdots + x_k))}
{\prod_{i\not = \ell, i<k} (x_i+u_i+ \cdots + x_{k-1}+u_{k-1} +x_k - (x_\ell+u_\ell+
  \cdots + x_k))}.
$$
Return now to~\eqref{T-C}. It is  easy to check that $\al_\ell(x,u)/(x_\ell+u_\ell+ \cdots + x_k+u_k)$ is
left unchanged by the exchange of $x_\ell$ and $x_{\ell+1}$. Consequently,
the sum of the lemma reads
$$
\sum_{\tau\in \Sn_k} \vareps(\tau) \ 
\tau  \left( T(x,u)\right)
=
\sum_{\tau\in \Sn_k} \vareps(\tau)\ \tau  \left( C(x,u)\right)
=
\det\left(\prod_{h=i}^{k-1}(x_j+u_h)\right)_{1\le i,j\le k}
=
\prod_{1 \le i <j \le k} (x_i-x_j),
$$
because $\prod_{h=i}^{k-1}(x_j+u_h)$ is a polynomial in 
$x_j$ of leading term
 $x_j^{k-i}$: the sum over
$\tau$ thus reduces to a Vandermonde determinant.
\end{proof}

\begin{Lemma}
\label{extr-si}
  Let $b=(b_1, \ldots, b_m) \in \ns^m$ and $\si \in \Sn_m$. Let
$$
\frac 1{x^e} =\frac {\si  (M(x))}{M(x) \si(x^b)},
$$
where as before $M(x)=x_2 \cdots x_m^{m-1}$. 
That is,  $e=(e_1, \ldots, e_m)$ where $e_i=b_{\tau(i)} -\tau(i)+i$ and
$\tau=\si^{-1}$.  
Let 
$k=\max\{i : b_i>0\}$ (if $b=(0, \ldots, 0)$, we take $k=0$).
Then
\beq\label{Lambda-1xe}
\Lambda \left( \frac 1 {x^e}\right)= 
\left\{
\begin{array}{cl}
\displaystyle   \prod_{i=1}^k {{e_i+ \cdots + e_k-1}\choose 
{e_i -1}},
 & \hbox{ if } \si(j)=j \hbox{ for all } j > k;\\
0, & \hbox{ otherwise }.
\end{array}\right.
\eeq
\end{Lemma}
\noindent{\bf Remark.} If  $\si(j)=j \hbox{ for all } j > k$, and $i\le k$, then 
$e_i+\cdots + e_k \ge 1$.    
Indeed, if $e_i+\cdots + e_k = b_{\tau(i)}+ \cdots + b_{\tau(k)}
+(i+\cdots +k) - (\tau(i)+\cdots + \tau(k))$ were non-positive, this
would mean that $\{\tau(i), \ldots, \tau(k)\}=\{i, \ldots, k\}$ and $b_i=\cdots =
b_k=0$, which contradicts the definition of $k$. 
However, $e_i$ may be non-positive, and in this case the above
expression vanishes. However, $e_i+k-i\ge 0$. When we apply this
lemma to prove Proposition~\ref{prop:eval-sum}, we will write the above product of
binomial coefficients in the following equivalent form:
\beq\label{prod-alt}
\frac{(e_1+\cdots +e_k)!}{\prod_{1\le i \le k}(e_i+k-i)!}
\prod_{i=1}^{k-1} \frac{e_i(e_i+1)\cdots (e_i+k-i)}{e_i+ \cdots + e_k}.
\eeq

\begin{proof}
  For an iterated Laurent series in $z_1, \ldots, z_k$, of the form 
$R(z)= \sum_{n\in \zs^k}  c(n_1, \ldots, n_k)  z_1^{n_1}\cdots z_k^{n_k}$, we
  define the \emm negative part, of $R(z)$ by
$$
[z^<] R(z)= [z_1^<\cdots z_k^<]R(z):=
\sum _{n_1<0, \ldots, n_k<0}  c(n_1, \ldots, n_k).
$$
Let us first prove that  if $f=(f_1, \ldots, f_k) \in \zs^k$,
\beq\label{neg-part}
[z^<] \left( \frac 1 {x^f}\right)=
\prod_{i=1}^k {{f_i+ \cdots + f_k-1}\choose 
{f_i -1}}.
\eeq
We adopt the standard convention that ${a\choose b}=0$ unless $0\le b\le a$. 
Given that $x_i=z_i-z_{i-1}$, we have
$$
\frac 1 {x^f} = \frac 1 {z_1^{f_1} \cdots z_k^{f_k} \left(
  1-\frac{z_1}{z_2}\right)^{f_2} \cdots \left(
  1-\frac{z_{k-1}}{z_k}\right)^{f_k}}.
$$
If  $f_i\le 0$ for some $i$, the $z$-expansion of $1/x^f$ only
involves non-negative 
powers of $z_i$, so that the negative part of $1/x^f$ is zero. The
right-hand side of~\eqref{neg-part} is zero as well, and
thus~\eqref{neg-part} holds.
Assume now $f_i>0$
for all $i$. The proof works by induction on $k$. If $k=1$ and $f_1>0$,
$$
[z^<] \left(\frac 1 {z_1^{f_1}}\right) = 1 = 
 {{f_1-1 }\choose {f_1-1 }}.
$$
For $k\ge 2$,
\begin{eqnarray*}
  [z^<] \left(\frac 1 {x^f}\right) &= &
[z^<] \frac 1 {
z_1^{f_1} z_2^{f_2}\left(  1-\frac{z_1}{z_2}\right)^{f_2} 
(z_3-z_2)^{f_3} \cdots (z_k-z_{k-1})^{f_k}}
\\
&=&
[z^<] \sum_{n\ge 0} {{n+f_2-1} \choose  {f_2-1}}
\frac{z_1^{n-f_1}}{z_2^{n+f_2}(z_3-z_2)^{f_3} \cdots (z_k-z_{k-1})^{f_k}} 
\\
&=&
 \sum_{n= 0}^{f_1-1} {{n+f_2-1} \choose  {f_2-1}}
[z_2^<\cdots z_k^{<}] \frac 1 {z_2^{n+f_2}(z_3-z_2)^{f_3} \cdots (z_k-z_{k-1})^{f_k}} 
\\
&=&
 \sum_{n= 0}^{f_1-1} {{n+f_2-1} \choose  {f_2-1}}
 {{n+ f_2+\cdots + f_k-1}\choose{n+f_2-1}}
\prod_{i=3}^k  {{f_i+ \cdots + f_k-1}\choose 
{f_i -1}}
\end{eqnarray*}
by the induction hypothesis. Now
\begin{multline*}
 \sum_{n= 0}^{f_1-1} {{n+f_2-1} \choose  {f_2-1}}
 {{n+ f_2+\cdots + f_k-1}\choose{n+f_2-1}}=\\
 \sum_{n= 0}^{f_1-1} {{n+ f_2+\cdots + f_k-1}\choose{n}}
{{ f_2+\cdots + f_k-1}\choose{f_2-1}}=
\\
 {{f_1+ f_2+\cdots + f_k-1}\choose{f_1-1}}
{{ f_2+\cdots + f_k-1}\choose{f_2-1}}.
\end{multline*}
The last equality results from the classical binomial identity
$$
\sum_{n=0}^a {{n+b}\choose n}= {{a+b+1}\choose a}.
$$
This gives~\eqref{neg-part}.

\medskip
Let us  now prove~\eqref{Lambda-1xe}.  Assume that $\si(j) = j$ for all $j>k$. This implies
that $e_{k+1}= \cdots = e_m=0$. As argued just
after the statement of the lemma, $e_k>0$. But then $1/x^e$ is
negative in $z_k$, and involves none of the variables $z_{k+1},
\ldots, z_m$. Thus, by definition of $\Lambda$,
$$
\Lambda \left( \frac 1 {x^e}\right) = [z_1^< \cdots z_k^<] \frac
1{x_1^{e_1} \cdots x_k^{e_k}} = \prod_{i=1}^k {{e_i+ \cdots + e_k-1}\choose 
{e_i -1}}
$$
(by~\eqref{neg-part}), and this gives the first part of~\eqref{Lambda-1xe}.

Assume now that  there exists $j>k$ such that
$\si(j)\not = j$. Then there also exists $j>k$ such that
$\si(j)<j$. Let us choose such a $j$.  Then there also exists
$\ell>\si(j)$ such that $\tau(\ell)<\ell$. We have
$$
\begin{array}{lllccccccccc}
e_{\si(j)}&=& b_j-j+\si(j)= -j +\si(j) <0,\\
e_\ell&=& b_{\tau(\ell)} -\tau(\ell) +\ell >0,
\end{array}
$$
with $\ell >\si(j)$. Let $\ell'=\max\{p >\si(j): e_p>0\}$ (this set is
non-empty as it contains $\ell$). Then
$e_{\ell'+1}=\cdots = e_m=0$, and  $1/x^e$ is non-negative in $z_{\si(j)}$
  but negative in $z_{\ell'}$. By definition of $\Lambda$, this
  implies that $\Lambda(1/x^e)=0$.
\end{proof}

\noindent{\em {Proof of Proposition~{\rm\ref{prop:eval-sum}.}}}
 Let us denote by
  $\SUM(b)$ the  sum we want to evaluate.  Let 
$k=\max\{i: b_i>0\}$. By   Lemma~\ref{extr-si}, the sum  can
 be reduced to \ps\ $\si$ that fix all points larger than $k$, and
 then we use the closed form expression~\eqref{prod-alt} of $ \Lambda \!
\left( \frac {\si ( M(x))}{ M(x) \si(x^b)}\right)$. This gives:
\begin{eqnarray*}
  \SUM(b)&=&\sum_{\si \in \Sn_k} \vareps(\si) \ \Lambda \!
\left( \frac {\si ( M(x))}{ M(x) \si(x^b)}\right)
\\
&=&\sum_{\tau\in \Sn_k} \vareps(\tau) \frac{(b_1+\cdots + b_k)!}
{\prod_{i=1}^k (b_{\tau(i)} -\tau(i)+k)!}\prod_{i=1}^{k-1}
\frac{(b_{\tau(i)} -\tau(i)+i) \cdots (b_{\tau(i)} -\tau(i)+k)}{b_{\tau(i)} -\tau(i)+i+ \cdots +
  b_{\tau(k)} -\tau(k)+k}\\
&=&\frac{(b_1+\cdots + b_k)!}{\prod_{i=1}^k (b_i-i+k)!}\prod_{1\le i <j \le k} (b_i-i-b_j+j).
\end{eqnarray*}
The last equality is the case $x_i= b_i-i$, $u_i=i$ of
Lemma~\ref{lem:identite}. 
 It is easy to check that, given that $b_{k+1}=\cdots = b_m=0$,
the above expression  coincides with 
$$
\frac{(b_1+\cdots + b_m)!}{\prod_{i=1}^m (b_i-i+m)!}\prod_{1\le i <j
  \le m} (b_i-i-b_j+j),
$$
as stated in Proposition~\ref{prop:eval-sum}.
\qed

\section{Final comments}
\label{sec:final}

Clearly, our proof of Theorem~\ref{thm:p}, dealing with permutations
of $\Sn^{(m)}$, is more complicated than
our proof of Theorem~\ref{thm:i}, dealing with involutions of
$\In^{(m)}$. We still wonder if there exists another change of
variables, another coefficient extraction or another way to perform
this extraction effectively that would simplify Sections~\ref{sec:extr-p}
and~\ref{sec:det-p}.

\medskip

Our approach is robust enough to be adapted to other
enumeration problems. Consider for instance the set $\tilde \Sn^{(m)}$  of
\ps\ $\pi$ of
$\Sn^{(m)}$ in which the values $1,2, \ldots, m$ occur in this
order. That is, $\pi^{-1}(1) < \cdots < \pi^{-1}(m)$. Garsia and
Goupil found recently a simple formula  for the number of such \ps\ of
(small) length $n$:  if $n\le 2m$, this number is~\cite[Coro.~6.2]{garsia-goupil}
\beq\label{gg}
\sharp\ \tilde \Sn^{(m)}_n= \sum_{r=0}^{n-m}(-1)^r {{n-m}\choose r } \frac{n!}{(m+r)!}.
\eeq
This was then reproved by Panova~\cite{panova}. See also~\cite{teddy}.

In terms of the generating tree described in Section~\ref{sec:tree-p}, this
means that one is only counting the nodes of the subtree
rooted at the \p\ $12\cdots m$. The only  change in the
functional equation of Proposition~\ref{prop:eq-p} is thus the initial
condition: instead of $1$ (which accounts for the empty \p), it is now
$v_1\cdots v_m t^m$. 
Sections~\ref{sec:inv-p} to~\ref{sec:extr-p} translate verbatim, and we reach the
following counterpart of Proposition~\ref{prop:p-extr}.
\begin{Proposition}
  The ordinary \gf\ of \ps\ that avoid \incm\ and in which the values
  $1, \ldots , m$ occur in this order is
obtained by applying the operator $\Lambda$ of Definition~{\rm\ref{def:Lambda}} to  a
rational function: 
 $$
\sum_{\tau \in \tilde \Sn^{(m)} } t^{|\tau|}
=
\Lambda\left(
\frac{t^m} 
{x_2^1\cdots x_m^{m-1}\,(1-t(\bx_1+\cdots +   \bx_m))}
\sum_{\si \in \Sn_m} \eps(\si)\, \si\! \left( 
\frac{x_2^1\cdots x_m^{m-1}}{\prod_{i=1}^m (x_1+\cdots +x_i)}
\right)\right),
$$
with $x_j=z_j-z_{j-1}$ and $z_0=0$.

Equivalently, the \emm exponential, \gf\ of these \ps\ is
$$
\sum_{\tau \in \tilde \Sn^{(m)} } \frac{t^{|\tau|}}{|\tau|!}
=
\Lambda\left(
\frac{t^me^{t(\bx_1+\cdots +   \bx_m)}} 
{x_2^1\cdots x_m^{m-1}}
\sum_{\si \in \Sn_m} \eps(\si)\, \si\! \left( 
\frac{x_2^1\cdots x_m^{m-1}}{\prod_{i=1}^m (x_1+\cdots +x_i)}
\right)\right).
$$
\end{Proposition}
We have not  further pursued in this direction, but one could try to
obtain a more explicit formula giving the whole \gf\ (which would
imply~\eqref{gg} when $n\le 2m$).

\medskip

 As discussed at the beginning of Section~\ref{sec:tree-inv}, the
generating tree for \incm-avoiding involutions can be described using
$m$ catalytic variables. Since these \invs \ are
equinumerous with \decm-avoiding involutions, it is natural to ask
whether one could derive
Theorem~\ref{thm:i} from this tree and the corresponding functional
equation. This could allow us to address the enumeration of
\incm-avoiding \emm fixed point free, \invs, for which
determinantal formulae exist (obtained by applying Gessel's $\theta$
operator~\cite{gessel-symmetric} to identities (5.41) and (5.42)
of~\cite{baik-rains}, which are equivalent to Theorem 2.3(3) in~\cite{okada}; see also Stanley's
survey~\cite[Thm.~8]{stanley-icm}). The  recursive construction we
have used  for involutions does not allow us to address this problem.

\bigskip
\bigskip
\noindent{\bf Acknowledgements.} The author thanks Aaron Jaggard for
communicating an early version of his paper with
Joseph Marincel~\cite{jaggard-marincel}.


\bibliographystyle{plain}
\bibliography{perm.bib}

\end{document}